\documentclass{balanced}
\usepackage{graphicx}
\usepackage{pstricks}
\usepackage{amscd}
\usepackage{amsmath}
\usepackage{amsxtra}

\usepackage{tikz}
\usetikzlibrary{arrows}
\usetikzlibrary{calc}
\usepackage{tikz-cd}

\usepackage{amsfonts,amssymb,amsmath}
\usepackage{url}
\usepackage{enumerate}

\newcommand{\strutstretchdef}{\newcommand{\strutstretch}{\vphantom{\raisebox{1pt}{$\big($}\raisebox{-1pt}{$\big($}}}}
\newtheorem{theorem}{Theorem}[section]
\newtheorem{lemma}[theorem]{Lemma}
\newtheorem{proposition}[theorem]{Proposition}
\newtheorem{corollary}[theorem]{Corollary}

\newtheorem{definition}[theorem]{Definition}
\newtheorem{problem}[theorem]{Problem}

\newtheorem{remark}[theorem]{Remark}
\renewcommand{\theenumi}{\roman{enumi}}

\numberwithin{equation}{section}

\newlength{\struh}
\newlength{\textminustop}
\strutstretchdef
\newcommand{\mf}{\mathfrak}

\newcommand*{\slam}{S_{\lambdab}}
\newcommand*{\slamc}[1]{S_{\lambdab_{\mf #1}}}
\newcommand*{\sib}[1]{\mathsf{sib}(#1)}

\newcommand*{\sibi}[2]{{\mathsf{sib}}_{#1}(#2)}

\newcommand*{\child}[1]{\mathsf{Chi}(#1)}

\newcommand*{\childnt}[2]{{\mathsf{Chi}}^{\ll#1\gg}(#2)}

\newcommand*{\childi}[2]{{\mathsf{Chi}}_{#1}(#2)}

\newcommand*{\parentn}[2]{{\mathsf{par}}^{\langle#1\rangle}(#2)}

\newcommand*{\parenti}[2]{{\mathsf{par}}_{#1}(#2)}

\newcommand*{\Ge}{\geqslant}
\newcommand*{\lambdab}{\boldsymbol\lambda}

\newcommand*{\Le}{\leqslant}
\newcommand*{\parent}[1]{\mathsf{par}(#1)}

\newcommand*{\rootb}{{\mathsf{root}}}

\usepackage{mathrsfs}
\usepackage{epsfig}
\usepackage{ifthen}

\newcommand{\ncom}{\newcommand}
\ncom{\bq}{\begin{equation}}
\ncom{\eq}{\end{equation}}
\ncom{\beqn}{\begin{eqnarray*}}
\ncom{\eeqn}{\end{eqnarray*}}
\ncom{\beq}{\begin{eqnarray}}
\ncom{\eeq}{\end{eqnarray}}
\ncom{\nno}{\nonumber}
\ncom{\rar}{\rightarrow}
\ncom{\Rar}{\Rightarrow}
\ncom{\noin}{\noindent}
\ncom{\bc}{\begin{centre}}
\ncom{\ec}{\end{centre}}
\ncom{\sz}{\scriptsize}
\ncom{\rf}{\ref}
\ncom{\sgm}{\sigma}
\ncom{\Sgm}{\Sigma}
\ncom{\dt}{\delta}
\ncom{\Dt}{Delta}
\ncom{\lmd}{\lambda}
\ncom{\Lmd}{\Lambda}
\ncom{\eps}{\epsilon}
\ncom{\pcc}{\stackrel{P}{>}}
\ncom{\dist}{{\rm\,dist}}
\ncom{\sspan}{{\rm\,span}}
\ncom{\im}{{\rm Im\,}}
\ncom{\sgn}{{\rm sgn\,}}
\ncom{\ba}{\begin{array}}
\ncom{\ea}{\end{array}}
\ncom{\eop}{\hfill{{\rule{2.5mm}{2.5mm}}}}
\ncom{\eoe}{\hfill{{\rule{1.5mm}{1.5mm}}}}
\ncom{\eof}{\hfill{{\rule{1.5mm}{1.5mm}}}}
\ncom{\hone}{\mbox{\hspace{1em}}}
\ncom{\htwo}{\mbox{\hspace{2em}}}
\ncom{\hthree}{\mbox{\hspace{3em}}}
\ncom{\hfour}{\mbox{\hspace{4em}}}
\ncom{\hsev}{\mbox{\hspace{7em}}}
\ncom{\vone}{\vskip 2ex}
\ncom{\vtwo}{\vskip 4ex}
\ncom{\vonee}{\vskip 1.5ex}
\ncom{\vthree}{\vskip 6ex}
\ncom{\vfour}{\vspace*{8ex}}
\ncom{\norm}{\|\;\;\|}
\ncom{\integ}[4]{\int_{#1}^{#2}\,{#3}\,d{#4}}
\ncom{\inp}[2]{\langle{#1},\,{#2} \rangle}
\ncom{\Inp}[2]{\Langle{#1},\,{#2} \Langle}
\ncom{\vspan}[1]{{{\rm\,span}\#1 \}}}
\ncom{\dm}[1]{\displaystyle {#1}}
\def\theequation{\thesection.\arabic{equation}}

\def \dep{\mathsf{d}}

\title[Classification of Drury-Arveson-type Hilbert modules]{Classification of Drury-Arveson-type Hilbert modules \\ associated with certain directed graphs}

\author[S. Chavan,  D. K. Pradhan \and S. Trivedi]{Sameer Chavan,  Deepak Kumar Pradhan \and Shailesh Trivedi
}

\extraline{The research of the second author was supported by the NBHM Research Fellowship, while the research of the third author was supported by the National Post-doctoral Fellowship (Ref. No. PDF/2016/001681), SERB.}

\classno{Primary 46E22, 32A25, 32A36; Secondary 47A13, 05C20, 28B20}

\begin{document}

\maketitle

\begin{abstract} Given a directed Cartesian product $\mathscr T$ of locally finite, leafless, rooted directed trees $\mathscr T_1, \ldots, \mathscr T_d$
of finite joint branching index, 
one may associate with $\mathscr T$ the Drury-Arveson-type $\mathbb C[z_1, \ldots, z_d]$-Hilbert module $\mathscr H_{\mathfrak c_a}(\mathscr T)$ of vector-valued holomorphic functions on the open unit ball $\mathbb B^d$ in $\mathbb C^d$,  where $a >0.$  In case all directed trees under consideration are without branching vertices, $\mathscr H_{\mathfrak c_a}(\mathscr T)$ turns out to be the classical Drury-Arveson-type Hilbert module $\mathscr H_{a}$ associated with the reproducing kernel $\frac{1}{(1 - \langle{z}, {w}\rangle)^a}$ defined on $\mathbb B^d$. 
Unlike the case of $d=1$, the above association does not yield a reproducing kernel Hilbert module if we relax the assumption that $\mathscr T$ has finite joint branching index.
The main result of this paper classifies all directed Cartesian product $\mathscr T$ for which the Hilbert modules $\mathscr H_{\mathfrak c_a}(\mathscr T)$ are isomorphic in case $a$ is a positive integer.  One of the essential tools used to establish this isomorphism is an operator-valued representing measure arising from $\mathscr H_{\mathfrak c_a}(\mathscr T).$ Further, a careful analysis of these Hilbert modules allows us to prove that the cardinality of the $k^{\tiny \mbox{th}}$ generation $(k =0, 1, \ldots)$ of $\mathscr T_1, \ldots, \mathscr T_d$ are complete invariants for $\mathscr H_{\mathfrak c_a}(\cdot)$ provided $ad \neq 1$. 
Failure of this result in case $ad =1$ may be attributed to the von Neumann-Wold decomposition for isometries. 
 Along the way, we identify the joint cokernel $E$ of the multiplication $d$-tuple $\mathscr M_{z}$ on $\mathscr H_{\mathfrak c_a}(\mathscr T)$ with orthogonal direct sum of tensor products of certain hyperplanes. In case $\mathscr T$ has finite joint branching index, this readily yields a neat formula for the dimension of $E$. 
\end{abstract}




\section{A classification problem}

In \cite{CPT}, we introduced and studied the notion of multishifts on the directed Cartesian product  of finitely many leafless, rooted directed trees. This was indeed an attempt to unify the theory of weighted shifts on rooted directed trees \cite{JJS} and that of classical unilateral multishifts \cite{JL}. Besides a finer 
analysis of various joint spectra and wandering subspace property of these multishifts, this work provided a scheme to associate a one parameter family of reproducing kernel Hilbert spaces $\mathscr H_{\mf c_a}(\mathscr T)~(a > 0)$ with every directed Cartesian product $\mathscr T$ of finite joint branching index,  see Corollary \ref{S-c-a-kernel-0} below (cf. \cite[Proposition 4.4]{AV} and \cite[Definition 4.1]{N}). 
These spaces consist of vector-valued holomorphic functions defined on the unit ball $\mathbb B^d$ in $\mathbb C^d$, and can be thought of as
{\it tree analogs} of the reproducing kernel Hilbert spaces $\mathscr H_{a}$ associated with the positive definite kernels \beq \label{c-D-A-k} \kappa_a(z, w):=\frac{1}{(1-\inp{z}{{w}})^a}, \quad z, w \in \mathbb B^d \eeq (refer to \cite{Dr} and \cite{Ar}; refer also to \cite{Z} for a comprehensive account of the theory of Hilbert spaces of holomorphic functions on the unit ball). 
Indeed, the reproducing kernels $\kappa_{\mathscr H_{\mf c_a}}(z, w)$ associated with $\mathscr H_{\mf c_a}(\mathscr T)$ are certain positive operator linear combinations of $\kappa_a(z, w)$ and multivariable hypergeometric functions (see \cite[Theorem 5.2.6]{CPT}). In particular, the Hilbert space $\mathscr H_a$ is {\it contractively contained} in $\mathscr H_{\mf c_a}(\mathscr T)$ (see \cite[Theorem 5.1]{PR}, cf. \cite[Proposition 4.5(1)]{AV}).
It is interesting to note that $\kappa_{\mathscr H_{\mf c_a}}(z, w)$ can be obtained by integrating certain perturbations of $\kappa_a(z, w)$ with respect to a finite family of spectral measures (see Remark \ref{int-s-k}).
Further, the Hilbert space $\mathscr H_{\mf c_a}(\mathscr T)$ carries a natural {\it Hilbert module} structure over the polynomial ring $\mathbb C[z_1, \ldots, z_d]$ with module action 
\beq \label{HM}
(p, h) \in \mathbb C[z_1, \cdots, z_d] \times \mathscr H_{\mf c_a}(\mathscr T) \longmapsto p(\mathscr M_{z})h \in \mathscr H_{\mf c_a}(\mathscr T),
\eeq
where $\mathbb C[z_1, \ldots, z_d]$ denotes the ring of polynomials in the complex variables $z_1, \ldots, z_d$ and $\mathscr M_z$ is the $d$-tuple of multiplication operators $\mathscr M_{z_1}, \ldots, \mathscr M_{z_d}$ acting on $\mathscr H_{\mf c_a}(\mathscr T)$ (refer to \cite[Section 2]{Sa} for the general theory of Hilbert modules over the algebra of polynomials).
We refer to $\mathscr H_{\mf c_a}(\mathscr T)$ as the {\it Drury-Arveson-type Hilbert module} associated with $\mathscr T$.

A thorough study of the Hilbert modules $\mathscr H_{\mf c_a}(\mathscr T)$ had been carried out in \cite[Chapter 5, Section 2]{CPT}. In particular, the essential normality of $\mathscr H_{\mf c_a}(\mathscr T)$ is shown to be closely related to the notion of finite joint branching index of $\mathscr T$ (see \cite[Proposition 5.2.9 and Example 5.2.20]{CPT}). 
In the present work, we continue our study of the Drury-Arveson-type Hilbert modules $\mathscr H_{\mf c_a}(\mathscr T)$. 
The investigations herein are motivated by the following classification problem for the Hilbert modules $\mathscr H_{\mf c_a}(\mathscr T)$ associated with the directed Cartesian product $\mathscr T$ (see \cite[Corollary 9.10]{ACJS}, \cite[Theorem 2.4]{CPT-1} for variants of this problem; see also \cite[Theorems 9.8 and 9.9]{ACJS}).  

\begin{problem} \label{prob}
For $j=1,2$, let $\mathscr T^{(j)}= (V^{(j)}, 
\mathcal{E}^{(j)})$ denote the directed Cartesian product of 
locally finite, leafless, rooted directed trees $\mathscr T^{(j)}_1, \ldots, \mathscr T^{(j)}_d$
of finite joint branching index. 
Under what conditions on $\mathscr T^{(1)}$ and $\mathscr T^{(2)}$, the Drury-Arveson-type Hilbert modules 
$\mathscr{H}_{\mf c_a}(\mathscr T^{(1)})$ and $\mathscr{H}_{\mf c_a}(\mathscr T^{(2)})$ are isomorphic\,? 
\end{problem}
Recall that the Hilbert modules $\mathscr{H}_{\mf c_a}(\mathscr T^{(1)})$ and $\mathscr{H}_{\mf c_a}(\mathscr T^{(2)})$ are {\it isomorphic} if there exists a unitary map $U : \mathscr{H}_{\mf c_a}(\mathscr T^{(1)}) \rar  \mathscr{H}_{\mf c_a}(\mathscr T^{(2)})$ such that $$U 
\mathscr{M}^{(1)}_{z_k} = \mathscr{M}^{(2)}_{z_k}U, \quad k=1, \ldots, d,$$
where $\mathscr M^{(j)}_{z_k}$ denotes the operator of multiplication by the coordinate function $z_k$ on $\mathscr{H}_{\mf c_a}(\mathscr T^{(j)})$ for $j=1,2$.
We refer to $U$ as a {\it Hilbert module isomorphism} between
 $\mathscr{H}_{\mf c_a}(\mathscr T^{(1)})$ and $\mathscr{H}_{\mf c_a}(\mathscr T^{(2)})$.

It turns out that for graph-isomorphic directed Cartesian products, the associated Drury-Arveson-type Hilbert modules are always isomorphic (see Remark \ref{graph-iso}). However, given any positive integer $k$, one can produce $k$ number of non-isomorphic directed Cartesian products for which the associated Drury-Arveson-type Hilbert modules are isomorphic (see Corollary \ref{m-non-iso}(ii)). Thus graph-isomorphism of directed Cartesian products is sufficient but not necessary to ensure the isomorphism of the associated Drury-Arveson-type Hilbert modules. This is in contrast with \cite[Theorem 2.11]{KK}, where countable directed graphs completely determine the associated tensor (quiver) algebras (up to Banach space isomorphism) (cf. \cite[Theorem 3.7]{S}).

The main result of this paper answers when two Drury-Arveson-type Hilbert modules $\mathscr H_{\mf c_a}(\mathscr T^{(j)})~(j=1, 2)$ are isomorphic in case $a$ is a positive integer (see Theorem \ref{main-thm} and Remark \ref{rmk-vW}). In particular, it provides complete unitary invariants for the Drury-Arveson-type Hilbert modules $\mathscr H_{\mf c_a}(\mathscr T)$ in terms of some discrete data associated with $\mathscr T$.  
Before we state this result, we need to reproduce several notions from \cite{JJS} and \cite{CPT} (the reader is advised to recall all the relevant definitions pertaining to the directed trees from \cite{JJS}).

\subsection{Multishifts on directed Cartesian product of directed trees}

We first set some standard notations.  
For a positive integer $d$ and a set $X$,  $X^d$ stands for the $d$-fold Cartesian product of $X$, while $\mbox{card}(X)$ stands the cardinality of $X$. 
The symbol ${\mathbb N}$ denotes the set of nonnegative
integers, and $\mathbb C$ denotes the field of complex numbers.  
For $\alpha =
(\alpha_1, \ldots, \alpha_d) \in {\mathbb{N}}^d,$ 
we use $\alpha!:=\prod_{j=1}^d \alpha_j!$ and $|\alpha|:=\sum_{j=1}^d
\alpha_j$. The modulus of a complex number $z$ is denoted by $|z|$. The complex conjugate of $z =(z_1, \ldots, z_d) \in \mathbb C^d$ is given by $\overline{z}:=(\overline{z}_1, \ldots, \overline{z}_d)$, while the Euclidean norm $(|z_1|^2 + \cdots + |z_d|^2)^{1/2}$ of $z$ is denoted by $\|z\|_2.$
The open ball in $\mathbb C^d$ centered at the origin and of
radius $r > 0$ is denoted by $\mathbb B^d_r$,
while the sphere 
centered at the origin and of radius $r > 0$ is denoted by $\partial
\mathbb B^d_r.$ 
For simplicity, the unit ball $\mathbb B^d_1$ and the
unit sphere $\partial \mathbb B^d_1$ are denoted respectively by
$\mathbb B^d$ and $\partial \mathbb B^d.$ 
Throughout this paper, we follow the standard conventions that the sum over the empty set is $0$, while the product over the empty set is always $1.$ 
 
For $j=1, \ldots, d,$ let $\mathscr T_j = (V_j, \mathcal E_j)$ be a leafless, rooted directed tree with root $\rootb_j$. 
The {\it directed Cartesian product} of $\mathscr T_1, \ldots, \mathscr T_d$ is the directed graph $\mathscr T = (V,\mathcal E)$ given by
\beqn
V := V_1 \times \cdots \times V_d, \eeqn
\vskip-.7cm
\beqn
\mathcal E := \Big\{(v, w) \in V \times V: 
~\mbox{there is a positive integer~} k \in \{1, \ldots, d\} \\ \mbox{such that}~ v_j = w_j~ \mbox{for}~ j \neq k~\mbox{and the edge}~ (v_k, w_k) \in \mathcal E_k \Big\},
\eeqn
where $v \in V$ is always understood as $v = (v_1, \ldots, v_d)$ with $v_j \in V_j,$ $j = 1, \ldots, d$. The $d$-fold directed Cartesian product of a directed tree $\mathscr T$ is denoted by $\mathscr T^d.$
The reader is referred to \cite[Chapter 2]{CPT} for the definitions of $\childi{j}{v}, \child{v}, \childnt{\alpha}{v}$,  $\parenti{j}{v},$ $\mathsf{Par}(v)$, $\sibi{j}{v}$ for a vertex $v \in V$, $\alpha \in \mathbb N^d,$ and $j=1, \ldots, d$. The {\it depth}  of a vertex $v \in V$ is the unique multiindex $\dep_v \in \mathbb N^d$ such that $$v \in \childnt{\dep_v}{\rootb},$$ where $\rootb$ denotes the root $(\rootb_1, \ldots, \rootb_d)$ of $\mathscr T.$ 
The depth of a vertex always exists (see \cite[Lemma 2.1.10(vi)]{CPT}).
For $k \in \mathbb N$, the set $$\mathcal G_k := \{v \in V : |\dep_v| = k\}$$ is referred to as the {\it $k^{\tiny th}$ generation} of $\mathscr T.$ 
A vertex $v \in V$ is called a {\it branching vertex} of $\mathscr T$ if $\text{card}(\child {v_j})$ $\geqslant  2$ for all $j=1, \ldots, d$.
The {\it branching index} $k_{\mathscr T}$ of $\mathscr T$ is the multiindex $(k_{\mathscr T_1}, \ldots, k_{\mathscr T_d}) \in \mathbb N^d$ given by
\[k_{\mathscr T_j}:=\begin{cases}
 1+\sup \big\{\dep_w : w\in V^{(j)}_{\prec}\big\}& \text{if $V^{(j)}_{\prec}$ is non-empty},\\
 0 & \text{otherwise},
\end{cases}
\]
where $V^{(j)}_{\prec}$ is the set of branching 
vertices of $\mathscr T_j$, $j=1, \ldots, d$. 
It is recorded in \cite[Proposition 2.1.19]{CPT} that
$$\childnt{k_{\mathscr T}}{V_\prec} \cap V_\prec = \emptyset,$$
where $V_{\prec}$ denotes the set of branching vertices of $\mathscr T.$

Let $\mathscr T= (V, 
\mathcal{E})$ be the directed Cartesian product of 
rooted directed trees $\mathscr T_1, \ldots, \mathscr T_d$ and let $V^{\circ}:=V \setminus \{\rootb\}$.
Consider the complex Hilbert space $l^2(V)$ of square summable complex functions on $V$
\index{$l^2(V)$}
equipped with the standard inner product. Note that $l^2(V)$ admits the orthonormal basis $\{e_v : v \in V\}$, 
where $e_v : V \rar \mathbb C$ denotes the indicator function of the set $\{v\},$ $v \in V$.
Given a system
$\lambdab=\{\lambda_j(v) : v \in V^{\circ}, ~ j=1, \ldots, d\}$ of positive numbers, we define
the {\it multishift $S_{\lambdab}$ on $\mathscr T$} 
with weights $\lambdab$ as the $d$-tuple of linear (possibly unbounded) operators  
$S_1, \ldots, S_d$ in $l^2(V)$ given by 
   \begin{align*}
   \begin{aligned}
{\mathscr D}(S_{j}) & := \{f \in l^2(V) \colon
\varLambda^{(j)}_{\mathscr T} f \in l^2(V)\},
   \\
S_{j} f & := \varLambda^{(j)}_{\mathscr T} f, \quad f \in {\mathscr
D}(S_{j}),
   \end{aligned}
   \end{align*}
where $\varLambda^{(j)}_{\mathscr T}$ is the mapping defined on
complex functions $f$ on $V$ by
   \begin{align*}
(\varLambda^{(j)}_{\mathscr T} f) (v) :=
   \begin{cases}
\lambda_j(v)  \cdot f\big(\parenti{j}{v}\big) & \text{if } v_j \in
V^\circ_j,
   \\
   0 & \text{otherwise}.
   \end{cases}
   \end{align*}
It is shown in \cite[Lemma 3.1.5]{CPT} that $S_j \in B(l^2(V))$ if and only if
\beq
\label{bdd}
\sup_{v \in V} \sum_{w \in \childi{j}{v}}  \lambda_j(w)^2 < \infty,
\eeq
where $B(\mathcal H)$ denotes the space of bounded linear operators on the Hilbert space $\mathcal H$.
Further, an examination of the proof of \cite[Proposition 3.1.7]{CPT} reveals that  for $1 \Le i, j \Le d,$ $S_iS_j=S_jS_i$ if and only if 
\beq \label{commuting} \lambda_j(u) \lambda_i(\parenti{j}{u})=\lambda_i(u) \lambda_j(\parenti{i}{u}), \quad u \in \mathsf{Chi}_j\childi{i}{v}, ~v \in V.\eeq
We say that $\slam$ is a {\it commuting multishift on} $\mathscr T$ if $\lambdab$ satisfies \eqref{bdd} and \eqref{commuting} for all $i, j=1, \ldots, d$. 

\begin{assumption*}
All the directed trees under consideration are countably infinite and leafless, that is, the cardinality of set of vertices is $\aleph_0$ and for every vertex $u$, $\mbox{card}(\child{u}) \Ge 1.$  
\end{assumption*}

For future reference, we reproduce from \cite[Proposition 3.1.7]{CPT} some general properties of commuting multishifts.

\begin{lemma} \label{ortho-v}
Let $\mathscr T= (V, 
\mathcal{E})$ be the directed Cartesian product of 
rooted directed trees $\mathscr T_1, \ldots, \mathscr T_d$ and
let $\slam$ be a commuting multishift on $\mathscr T.$
Then, for any $\alpha \in \mathbb N^d,$  the following statements hold:
\begin{enumerate}
\item[(i)] $S^{*\alpha}_{\lambdab} S^{\alpha}_{\lambdab}$ is a diagonal operator $($with respect to the orthonormal basis $\{e_v\}_{v \in V})$ with diagonal entries $\|\slam^{\alpha}e_v\|^2,$ $v \in V$.
\item[(ii)]
For distinct vertices $v, w \in V,$ $\inp{S^{\alpha}_{\lambdab}e_v}{S^{\alpha}_{\lambdab}e_w}=0.$
\end{enumerate}
\end{lemma}

Let
$\slam=(S_1, \ldots, S_d)$ be a commuting multishift on $\mathscr T$ with weight system $\lambdab$. Assume that $\slam$ is {\it joint left invertible}, that is, $\sum_{j=1}^d S^*_jS_j$ is invertible.
Then the {\it spherical Cauchy dual}  $S^{\mathfrak s}_{\lambdab} = (S_1^{\mathfrak s}, \ldots, S_d^{\mathfrak s})$ of  $S_{\lambdab}$ is given by
$$S_j^{\mf s} e_v = \Big(\sum_{i=1}^d \|S_i e_v\|^2 \Big)^{-1} \sum_{w \in \childi{j}{v}} \lambda_j(w) e_w, \quad v \in V,~j=1, \ldots, d.$$ 
Note that $S^{\mathfrak s}_{\lambdab}$ is the multishift on $\mathscr T$ with weights 
\beq \label{dual-wts}
{\lambda_j(w)}\Big({\sum_{i=1}^d \|S_i e_{v}\|^2}\Big)^{-1}, \quad w \in \childi{j}{v},~v \in V, ~j=1, \ldots, d.\eeq
In general, $S^{\mathfrak s}_{\lambdab}$ is not commuting (see \cite[Proposition 5.2.10]{CPT}).
However, if $S^{\mathfrak s}_{\lambdab}$ is commuting, then it is a joint left invertible commuting multishift such that $(\slam^{\mf s})^{\mf s}=\slam.$ 

\subsection{Joint cokernel of multishifts}

The main result of this paper relies heavily on the description of the joint cokernel $\ker \mathscr M^*_z$ of the multiplication tuple $\mathscr M_z$ acting on the Drury-Arveson-type Hilbert space $\mathscr H_{\mf c_a}(\mathscr T)$. The first step in this direction is to realize $\ker \mathscr M^*_z$ as the solution space of certain systems of linear equations arising from the eigenvalue problem for the adjoint of a commuting multishift. For this realization, we find it necessary to collect required graph-theoretic jargon as introduced in \cite[Chapter 4]{CPT}. 

For a set $A,$ let $\mathscr P(A)$ denote the set of all subsets of $A.$ In the case when $A= \{ 1, \cdots, d\}$, we simply write $\mathscr P$ in place of $\mathscr P(A)$. 
Let $\mathscr T = (V,\mathcal E)$ be the directed Cartesian product of rooted directed trees $\mathscr T_1, \cdots, \mathscr T_d$.
Consider the set-valued function $\Phi : \mathscr P \rar \mathscr P(V)$ given by 
\beq \label{phi-F-eqn}
\Phi(F) = \Phi_F := \big\{v \in V : v_j \in V^{\circ}_j~\mbox{if~}j \in F,~\mbox{and}~v_j =\mathsf{root}_j~\mbox{if}~j \notin F \big\},~ F \in \mathscr P,
\eeq
where $V^{\circ}_j := V_j \setminus \{\rootb_j\},$  $j=1, \ldots, d.$ 
Note that 
\begin{enumerate}
\item[$\bullet$]
if $F \neq G,$ then $\Phi_F \cap \Phi_G = \emptyset$,  
\item[$\bullet$] if $v \in V$, then $v \in \Phi_F$ for $F:=\{j \in \{1, \cdots, d\} : v_j \neq \mathsf{root}_j\}.$ 
\end{enumerate}
Thus it follows that
\beq \label{V-phi-F} V  = \displaystyle \bigsqcup_{F \in \mathscr{P}} \Phi_F~(\mbox{disjoint sum}).\eeq 
For $F \in \mathscr P$ and $u \in \Phi_F$, define \beqn \label{sib}
\mathsf{sib}_F(u) := \begin{cases} \mathsf{sib}_{i_1}  \mathsf{sib}_{i_2} \cdots \mathsf{sib}_{i_k}(u) & \mbox{if~} F= \{i_1, \cdots, i_k \}, \\
\{u\} & \mbox{if~} F = \emptyset.
\end{cases}
\eeqn
Define an equivalence relation $\sim$ on $\Phi_F$ by $$u \sim v ~\mbox{if and only if}~ u \in \mathsf{sib}_F(v),$$
and note that 
for any $u \in \Phi_F,$ the equivalence class containing $u$ is precisely $\mathsf{sib}_F(u)$. 
An application of the axiom of choice allows us to form a set $\Omega_F$ by picking up exactly one element from each of the equivalence classes $\mathsf{sib}_F(u),$ $u \in \Phi_F$. We refer to $\Omega_F$ as an {\it indexing set corresponding to $F$}.
Thus we have the disjoint union \beq \label{phi-F} \Phi_F = \displaystyle \bigsqcup_{u \in \Omega_F} \mathsf{sib}_F(u).\eeq
This combined with \eqref{V-phi-F} yields
the following decomposition of $l^2(V):$
\beq \label{deco}
l^2(V) = \bigoplus_{F  \in \mathscr{P}}\bigoplus_{u \in \Omega_{F}} l^{2}(\mathsf{sib}_{F}(u)).
\eeq 
For $F \in \mathscr P$ and $v = (v_1,\cdots, v_d) \in V$, let $v_F \in V$ denote the $d$-tuple with $j^{\tiny \mbox{th}}$ coordinate,  $1 \Le j \Le d,$ given by $$ (v_F)_j= \begin{cases} 
v_j & \mbox{~if~} j \in F,  \\
\mathsf{root}_j & \mbox{~if~} j \notin F. \end{cases}$$ 
Further, for $i=1, \ldots, d$ such that $i \notin F$ and $u_i \in V_i$, we
define $v_F|u_i \in V$ to be the $d$-tuple $(w_1, \cdots, w_d)$, where  
\beqn
w_j = \begin{cases} 
u_i & \mbox{~if~} j=i,  \\
(v_F)_j & \mbox{~otherwise}. \end{cases}
\eeqn
For $F, G \in \mathscr P$ such that $G \subseteq F$ and $u \in \Phi_F,$
define 
\beqn \label{sib-F-G-u}
\mathsf{sib}_{F, G}(u):=\{v_G : v \in \mathsf{sib}_F(u)\}. 
\eeqn
In view of \eqref{deco}, it can be deduced from \cite[Lemma 4.1.6]{CPT} that the joint kernel  $E:=\bigcap_{j=1}^d\ker S^*_j$ of $S_{\lambdab}^*$ is given by 
\beq \label{j-kernel-1}
E = [e_\rootb] \oplus \bigoplus_{\underset{F \neq \emptyset}{F  \in \mathscr{P}}} \bigoplus_{u \in \Omega_{F}} \mathcal L_{u, F},
\eeq 
where  
$\mathcal L_{u, F} \subseteq l^2(\mathsf{sib}_F (u)) $ is the solution space of the following system of equations

\beq \label{system-main} 
{\sum_{w \in \mathsf{sib}_j(v_G | u_j)} f(w) \lambda_{j}(w) =0,\quad j \in F, ~\ v_G \in \mathsf{sib}_{F, G}(u)} \mbox{ and }  G = F \setminus \{ j \}
\eeq 
(see the discussion following \cite[Lemma 4.1.6]{CPT} for more details).
The number of variables $M_{u, F}$ and number of equations $N_{u, F}$ in the above system are given by 
\beq
\label{MuF-NuF}
M_{u, F}=\mbox{card} 
(\mathsf{sib}_{F}(u)) = \prod_{j \in F} \mbox{card} (\mathsf{sib}(u_j)),
\quad
N_{u, F}=\sum_{j \in F } \prod_{\underset{i \neq j}{i \in F}}\mbox{card}(\mathsf{sib}_i
(u)).
\eeq 
In particular, $\mathcal L_{u, F}$ is finite dimensional whenever the directed trees $\mathscr T_1, \ldots, \mathscr T_d$ are locally finite. 
Indeed, $M_{u, F}$ and $N_{u, F}$ are finite in this case.

We present the following useful lemma for future reference.
\begin{lemma} \label{lem-E-finite}
Let $\mathscr T= (V, 
\mathcal{E})$ be the directed Cartesian product of 
locally finite, rooted directed trees $\mathscr T_1, \ldots, \mathscr T_d$  and let
$\slam$ be a commuting multishift on $\mathscr T$.
Then the joint kernel  $E$ of $S_{\lambdab}^*$ is given by 
\beq \label{j-kernel-2} 
E = \bigoplus_{{F  \in \mathscr{P}}} \bigoplus_{u \in \Omega_{F}} \mathcal L_{u, F},
\eeq 
where  
$\mathcal L_{u, F} \subseteq l^2(\mathsf{sib}_F (u)) $ is the solution space of the system \eqref{system-main}.
Moreover, if $\mathscr T$ is of finite joint branching index, then 
\begin{enumerate}
\item[(i)] for any $F \in \mathscr P$,  $\mathcal L_{u, F} \neq \{0\}$ for at most finitely many $u \in \Omega_F$, and 
\item[(ii)] $E$ is finite dimensional.
\end{enumerate}
\end{lemma}
\begin{proof}
Note that $\Omega_{\emptyset}=\{ \rootb \}$. Thus the system \eqref{system-main}  is vacuous, and hence $\mathcal L_{\rootb, \emptyset} =  [e_{\rootb}]$.
The desired expression for $E$ is now obvious from \eqref{j-kernel-1}. The part (ii) is immediate from \cite[Corollary 3.1.14]{CPT}, while (i) is clear in view of \eqref{j-kernel-2} and (ii).
\end{proof}

\subsection{Statement of the main result}

We recall from \cite[Theorem 5.2.6]{CPT} that $\mathscr H_{\mf c_a}(\mathscr T)$ is a reproducing kernel Hilbert space of $E$-valued holomorphic functions defined on the open unit ball $\mathbb B^d$ in $\mathbb C^d$. The reproducing kernel $\kappa_{\mathscr H_{\mf c_a}(\mathscr T)} : \mathbb B^d \times \mathbb B^d \rar B(E)$ associated with $\mathscr H_{\mf c_a}(\mathscr T)$ is given by 
\beqn
 \kappa_{\mathscr H_{\mf c_a}(\mathscr T)}(z, w) =  \sum_{{F  \in \mathscr{P}}} \sum_{u \in \Omega_F} \Big (\sum_{\alpha \in \mathbb N^d}  \frac{\dep_{u}!}{(\dep_{u}+\alpha)!}\, {\prod_{j=0}^{|\alpha|-1}(|\dep_u|+a + j)}\,   z^{\alpha} \overline{w}^{\alpha} \Big) P_{\mathcal L_{u, F}}, \quad z, w \in \mathbb B^d, 
\eeqn
where $P_{\mathcal L_{u, F}}$ is the orthogonal projection on $\mathcal{L}_{u,F}$ (see \eqref{j-kernel-2}). 

We are now ready to state the main result of this paper.
\begin{theorem}\label{main-thm}
Let $a, d$ be positive integers such that $ad \neq 1,$
and fix $j=1,2$. Let $\mathscr T^{(j)}= (V^{(j)}, 
\mathcal{E}^{(j)})$ be the directed Cartesian product of 
locally finite rooted directed trees 
$\mathscr T^{(j)}_1, \ldots, \mathscr T^{(j)}_d$
of finite joint branching index. 
Let $\mathscr H_{\mf c_a}(\mathscr T^{(j)})$ be the Drury-Arveson-type Hilbert module associated with $\mathscr T^{(j)}$. 
Let 
$E^{(j)}$ be the subspace of constant functions in $\mathscr H_{\mf c_a}(\mathscr T^{(j)})$
and let $\mathcal{L}^{(j)}_{u,F}$ be as appearing in the decomposition \eqref{j-kernel-2} of $E^{(j)}$. 
Then the following statements are equivalent:
\begin{enumerate}
\item[(i)] The Hilbert modules $\mathscr{H}_{\mf c_a}(\mathscr T^{(1)})$ and $\mathscr{H}_{\mf c_a}(\mathscr T^{(2)})$ are isomorphic.
\item[(ii)] For any $ \alpha \in \mathbb{N}^d$ and $F \in \mathscr P$,
\beqn \label{main-cond} \displaystyle \sum_{\underset{\dep_u=\alpha}{u \in \Omega_{F}^{(1)}}} \dim\mathcal{L}^{(1)}_{u,F} = \displaystyle \sum_{\underset{\dep_v=\alpha}{v\in \Omega_{F}^{(2)}}} \dim \mathcal{L}^{(2)}_{v,F}. \eeqn
\item[(iii)] For any $n \in \mathbb{N}$ and $l=1, \ldots, d$,  
\beqn
\sum_{\underset{\dep_u=n \epsilon_l}{u \in \Omega^{(1)}_{\{l\}}}} (\mbox{card}( \mathsf{sib}_l(u)) -1) &=& \sum_{\underset{\dep_u=n \epsilon_l}{u \in \Omega^{(2)}_{\{l\}}}} (\mbox{card}( \mathsf{sib}_l(u)) -1),
 \eeqn
where $\epsilon_l$ is the $d$-tuple with $1$ in the $l^{\mbox{\tiny th}}$ place and zeros elsewhere.
\item[(iv)] For any $n \in \mathbb{N}$ and $l=1, \ldots, d$,
\beqn \mbox{card}(\mathcal{G}_{n}(\mathscr T^{(1)}_l))=\mbox{card}(\mathcal{G}_{n}(\mathscr T^{(2)}_l)), \eeqn
where $\mathcal{G}_{n}(\mathscr T^{(j)}_l)$ is the $n^{th}$ generation of $\mathscr T^{(j)}_l$, $j=1,2$. 
 \end{enumerate}
\end{theorem}
\begin{remark} \label{rmk-vW}
The above result does not hold true in case $ad=1$. This may be attributed to the von Neumann-Wold decomposition for isometries \cite[Chapter I]{Co} $($see the discussion following \cite[Problem 2.3]{CPT-1}$)$.
In case $d=1$, (iv) is equivalent to the following: 
\beq \label{constant-g} \sum_{v \in V^{(1)}_{\prec} \cap\, \mathcal G_n(\mathscr T^{(1)})}\Big(\mbox{card}(\child{v})-1\Big)=\sum_{v \in V^{(2)}_{\prec} \cap \mathcal\,G_n(\mathscr T^{(2)})}\Big(\mbox{card}(\child{v})-1\Big).\eeq
In particular, Theorem \ref{main-thm} can be seen as a multivariable counterpart of $k^{\tiny \mbox{th}}$
generation branching degree as defined in \cite[Equation (92)]{ACJS} $($cf.\cite[Theorem 5.1]{CPT-1}$)$.
Further, it is evident from the equivalence of (i) and (iv) above that non-graph-isomorphic directed Cartesian products can yield isomorphic Drury-Arveson-type Hilbert modules. Finally, note that the operator theoretic statements (i) and (ii) are equivalent to purely graph theoretic statements (iii) and (iv).
\end{remark}

We discuss here some immediate consequences of Theorem \ref{main-thm}.
Recall that two directed graphs are {\it isomorphic} if there exists a bijection between their sets of vertices which preserves (directed) edges.
\begin{corollary}
Let $a, d$ be positive integers and let $\mathscr T= (V, 
\mathcal{E})$ be the directed Cartesian product of 
locally finite rooted directed trees $\mathscr T_1, \ldots, \mathscr T_d$ 
of finite joint branching index. 
Then
the Drury-Arveson-type Hilbert module $\mathscr H_{\mf c_a}(\mathscr T)$ associated with $\mathscr T$ is isomorphic to the classical Drury-Arveson-type Hilbert module $\mathscr H_{a}$ if and only if for any $j=1, \ldots, d,$ the directed tree $\mathscr T_j$ is graph isomorphic to the rooted directed tree without any branching vertex. 
\end{corollary}
\begin{proof}
The sufficiency part is immediate from \cite[Remark 3.1.1]{CPT}, while the necessary part follows from the equivalence of (i) and (iv) of Theorem \ref{main-thm}, and the fact that a rooted directed tree without any branching vertex is unique up to graph isomorphism.
\end{proof}

\begin{corollary} \label{m-non-iso}
Given positive integers $a$ and $d$, we have the following statements: 
\begin{enumerate}
\item[(i)] There exist infinitely many mutually non-isomorphic  Drury-Arveson-type Hilbert modules. 
\item[(ii)] Given any positive integer $k,$ there exist $k$ number of mutually non-isomorphic  directed Cartesian products $\mathscr T=(V, \mathcal E)$ of locally finite rooted directed trees $\mathscr T_1, \ldots, \mathscr T_d$ such that the associated Drury-Arveson-type Hilbert modules $\mathscr H_{\mf c_a}(\mathscr T)$ are isomorphic.
\end{enumerate}
\end{corollary}
\begin{proof} 
We need the following example of rooted directed tree discussed in \cite[Chapter 6]{JJS}.
For a positive integer $n_0$, consider the directed tree 
$\mathscr T_{n_0, 0}=(V, \mathcal E)$ as follows:
\beqn
V =\mathbb N,  \quad \mathcal E =\{(0, j) : j =1, \ldots, n_0\}  \cup  \bigcup_{j=1}^{n_0} \big\{(j + (l-1)n_0, j+ln_0) : l \Ge 1 \big\}.
\eeqn

(i) Consider the directed Cartesian product \beqn \mathscr T^{(k)}:=\mathscr T_{k, 0} \times \mathscr T^{d-1}_{1, 0}, \quad k \Ge 1.\eeqn
It is now immediate from Theorem \ref{main-thm} that the Drury-Arveson-type Hilbert modules $\mathscr H_{\mf c_a}(\mathscr T^{(k)})$ associated with the directed Cartesian product $\mathscr T^{(k)}$, $k \Ge 1$ are mutually non-isomorphic.

(ii) Fix a positive integer $k$. 
For $j=1, \ldots, k,$ consider the rooted directed tree $\mathscr T_{1j}$ with $\child{\rootb}=\{u, v\}$, $\mbox{card}(\child{u})=2k-j$, $\mbox{card}(\child{v})=j,$ and  $\mbox{card}(\child{w})=1$ for all remaining vertices $w$ in $\mathscr T_{1j}$. 
Consider the directed Cartesian product \beqn \mathscr T^{(j)}:=\mathscr T_{1j} \times \mathscr T^{d-1}_{1, 0}, \quad j=1, \ldots, k.\eeqn
Then $\mathscr T^{(j)},$ $1 \Le j \Le k$ are mutually non-isomorphic. Now apply Theorem \ref{main-thm} to obtain the desired conclusion in (ii).
\end{proof}

Our proof of Theorem \ref{main-thm} occupies substantial part of this paper. It is fairly long and quite involved as compared to the case of $d=1$ (see \cite[Theorem 5.1]{CPT-1}). It is worth mentioning that in case $d=1$, the conclusion of Theorem \ref{main-thm} holds without the assumption of finite branching index. An essential reason for this is the fact that any left-invertible analytic operator admits an analytic model in the sense of \cite{Sh}. On the other hand, in dimension $d \Ge 2,$ there is no successful counterpart of Shimorin's construction of an analytic model for joint left-invertible tuples (cf. \cite[Theorem 4.2.4, Remark 4.2.5 and Appendix]{CPT}). 
The proof of the main theorem is comprised of two parts, namely, Sections 2 and 3. 
Here is a brief overview of these sections. 
\begin{enumerate}
\item[$\bullet$] In Section 2, we introduce and study a one parameter family $\mathcal S_{\mathscr T}$ of spherically balanced multishifts $\slamc{c}$. 
This is carried out in three subsections.

\begin{enumerate}
\item[$\diamond$] The first subsection is devoted to an elaborated description of joint cokernel $E$  of multishifts in the family $\mathcal S_{\mathscr T}$. It turns out that the building blocks $\mathcal L_{u, F}$ appearing in the decomposition \eqref{j-kernel-2} of $E$ can be identified with tensor product of certain hyperplanes (see Theorem \ref{dimE-thm}). These hyperplanes further can be looked upon as the kernel of row matrices with all entries equal to $1$ and of size dependent on coordinate siblings of $u$.  This description readily provides a neat formula for the dimension of $E$ (see Corollary \ref{dimE}).
\item[$\diamond$]
In the second subsection,
we show that the multishifts in $\mathcal S_{\mathscr T}$ can be modeled as multiplication $d$-tuples on reproducing kernel Hilbert spaces $\mathscr H_{\mf c}(\mathscr T)$ of vector-valued holomorphic functions defined on a ball in $\mathbb C^d$ (see Theorem \ref{S-c-a-kernel}). We also provide a compact formula for the reproducing kernel associated with $\mathscr H_{\mf c}(\mathscr T)$ (see \eqref{rk-formula}). In particular, these results apply to Drury-Arveson-type multishifts $\slamc{c_a}$ and their spherical Cauchy dual tuples $\slamc{c_a}^{\mf s}$.
The sequences $\mf c$ associated with $\slamc{c_a}$ and $\slamc{c_a}^{\mf s}$ are given respectively by
\beq
\label{c-a}
\mf c(t) := \mf c_a(t)=\frac{t+d}{t+a}, \quad \mf c(t) :=\frac{1}{\mf c_a(t)}, \quad t \in \mathbb N,
\eeq
where $a$ is a positive real number. We emphasize that if we relax the assumption that $\mathscr T$ is of finite joint branching index, then the above model theorem fails unless the dimension $d=1$ (see Remark \ref{int-s-k}).
\item[$\diamond$]
In the last subsection, we  introduce and study the notion of an operator-valued representing measure  for the Hilbert module $\mathscr H_{\mf c}(\mathscr T)$.
Existence of a representing measure for $\mathscr H_{\mf c}(\mathscr T)$ is shown to be equivalent to the assertion that $\{\prod_{j=0}^{n-1} \mf c(j)\}_{n \in \mathbb N}$ is a Hausdorff moment sequence (see Theorem \ref{slice-rep-gen}). As an application, we show that Drury-Arveson-type Hilbert module $\mathscr H_{\mf c_a}(\mathscr T)$ admits a representing measure if $a$ is an integer bigger than or equal to $d$. We also show that $\mathscr H^{\mf s}_{\mf c_a}(\mathscr T)$ (model space of $\slamc{c_a}^{\mf s}$) admits a representing measure if $a$ is a positive integer less than $d$. Further, we explicitly compute the representing measures in both these situations (see Corollaries \ref{coro-1-slice-rep} and \ref{coro-2-slice-rep}).
\end{enumerate}
\item[$\bullet$] In Section 3, we prove the main theorem. This section begins with the observation that the classification of  Drury-Arveson-type Hilbert modules $\mathscr H_{\mf c_a}(\mathscr T)$ is equivalent to the unitary equivalence of  operator-valued representing measures of $\mathscr H_{\mf c_a}(\mathscr T)$ (resp. $\mathscr H^{\mf s}_{\mf c_a}(\mathscr T)$) if $a \Ge d$  (resp. if $a < d$) (see Lemma \ref{set-measure-lem}).  We then establish another key observation that any isomorphism between two Drury-Arveson-type Hilbert modules preserves the decomposition \eqref{j-kernel-2} of $E$ over each generation (see Proposition \ref{level-eq}). Finally, we put all the pieces together to obtain a proof of Theorem \ref{main-thm}.
\end{enumerate}

A strictly higher dimensional fact in graph theory ({\it constant on parents is constant on generations}) and closely related to the notion of spherically balanced multishift is added as an appendix (see Theorem A).


\section{A family of spherically balanced multishifts}

Let $\mathscr T = (V,\mathcal E)$ be the directed Cartesian product of locally finite rooted directed trees $\mathscr T_1, \cdots, \mathscr T_d$. 
Given a sequence $\mf c : \mathbb{N} \rar (0, \infty)$, we can associate  a system $\lambdab_{\mf c}=\{\lambda_j(w) : w \in V^{\circ}, ~j=1, \ldots, d\}$ to $\mathscr T$ as follows:
\beq \label{sp-wts}
\lambda_{j}(w) = \sqrt{\frac{\mf c(|\dep_v|)}{{\mbox{card}(\childi{j}{v})}}} \sqrt{\frac{\dep_{v_j}  + 1}{|\dep_v| + d}}, \quad w \in \childi{j}{v},~v \in V,~ j=1, \ldots, d,
\eeq
where $\dep_v$ denotes the depth of the vertex $v$ in $\mathscr T$.
 Note that
 \beqn
\sup_{v \in V} \sum_{w \in \childi{j}{v}} \lambda_{j}(w)^2 = \sup_{v \in V} \mf c(|\dep_v|) \frac{\dep_{v_j}  + 1}{|\dep_v| + d}.
\eeqn 
It follows that the multishift $S_{\lambdab_{\mf c}}$ with weights $\lambdab_{\mf c}$ is bounded if and only if the sequence $\mf c$ is bounded (see \eqref{bdd}). In this case,  
as shown in \cite[Proposition 5.2.3]{CPT}, the multishift 
$S_{\lambdab_{\mf c}}$ turns out to be commuting and spherically balanced. Recall that a commuting multishift $S_{\lambdab}=(S_1, \cdots, S_d)$ is {\it spherically balanced} if the function $\mf C : V \rar (0, \infty)$ given by
\beqn  \mf C(v) := \sum_{j=1}^d \|S_j e_v\|^2, \quad v \in V \eeqn
is
constant on every generation $\mathcal G_t$, $t \in  \mathbb N$. In case $\slam =\slamc{c},$ \beq 
\label{constant-gen}
\mf C(v)=\mf c(|\dep_v|), \quad v \in V. \eeq 
In case the directed trees $\mathscr T_j$ are without branching vertices, $\slamc{c}$ is spherical (or homogeneous with respect to the group of unitary $d \times d$ matrices) in the sense of \cite[Definition 1.1]{CY}. This may be concluded from \cite[Theorem 2.1]{CY}.
On the other hand, for a spherically balanced multishift $\slam$, the spherical Cauchy dual tuple $\slam^{\mf s}$ is always commuting. In fact, by \cite[Proposition 5.2.10]{CPT}, $\slam^{\mf s}$ is commuting if and only if 
\beq \label{w-sp}
\mf C~\mbox{is constant on~} \mathsf{Par}(v):=\bigcup_{j=1}^d \parenti{j}{v}~\mbox{for all~}v \in V^{\circ}.
\eeq
In dimension bigger than $1$, there is a curious fact that the apparently weaker condition \eqref{w-sp} implies that $\mf C$ is constant on every generation $\mathcal G_t,$ $t \in \mathbb N$. 
It follows that if  $d \Ge 2$, then $S_{\lambdab}$ is a spherically balanced $d$-tuple if and only if its spherical Cauchy dual $\slam^{\mf s}$ is commuting. 
Since the above facts play no essential role in the main result of this paper, we relegate its proof to an appendix. 
Needless to say, these facts are strictly higher dimensional.

The following family of multishifts plays a central role in the present investigations:  
\beq
\label{family}
\mathcal S_{\mathscr T} :=\{S_{\lambdab_{\mf c}} : \inf {\mf c} >0, ~\sup {\mf c} < \infty\}.
\eeq
Our proof of Theorem \ref{main-thm} is based on a thorough study of this family.
This includes a dimension formula for joint cokernel, an analytic model and  existence of operator-valued representing measures for multishifts in this family.

\subsection{A dimension formula}

In this subsection, we obtain a neat formula for the dimension of joint cokernel of members $\slamc{c}$ belonging to the family $\mathcal S_{\mathscr T}$. It is worth noting that this formula is independent of $\mf c$ due to the specific form of the weight system $\lambdab_{\mf c}$ of multishifts from $\mathcal S_{\mathscr T}$ (see Lemma \ref{system-main-Sc-lem} below).
First a definition (recall all required notations from Subsection 1.2).

Fix a nonempty $F \in \mathscr P$ and let $u \in \Omega_F$. 
 For $j \in F,$ define the linear functional $X_j : l^2(\mathsf{sib}(u_j)) \rar \mathbb C$ 
by 
\beq \label{def-Xj}
X_j(f)=\displaystyle \sum_{\eta \in \mathsf{sib}(u_j)}f(\eta), \quad f \in l^2(\mathsf{sib}(u_j)).\eeq
The description of the joint cokernel for a member of $\mathcal S_{\mathscr T}$ is intimately related to the kernel of the linear functionals $X_j,$ $j \in F.$
\begin{lemma} 
\label{system-main-Sc-lem} 
Let $\mathscr T = (V,\mathcal E)$ be the directed Cartesian product of locally finite rooted directed trees $\mathscr T_1, \cdots, \mathscr T_d$ 
and let $S_{\lambdab_{\mf c}}$ be a multishift belonging to the family $\mathcal S_{\mathscr T}$. Let $\mathcal{L}_{u,F}$ be as appearing in \eqref{j-kernel-2}.
Then, for $F \in \mathscr P$ and $u \in \Omega_F,$ the following are equivalent:
\begin{enumerate}
\item[(i)] $f$ belongs to $\mathcal L_{u, F}$.
\item[(ii)] 
$ \displaystyle
\sum_{w \in \mathsf{sib}_j(v_G | u_j)} f(w)  =0$ for any $j \in F,$  $v_G \in \mathsf{sib}_{F, G}(u) \mbox{ with }  G = F \setminus \{ j \}.$ 
\item[(iii)] $\displaystyle \sum_{\tiny \eta \in \mathsf{sib}(u_{j})} f(v_G|\eta) e_{\eta} \in \ker X_{j}$ for any $j \in F,$  $v_G \in \mathsf{sib}_{F, G}(u) \mbox{ with }  G = F \setminus \{j\},$ where $X_j$ is as given in \eqref{def-Xj}.
\end{enumerate}
\end{lemma}
\begin{proof}
By \eqref{sp-wts}, for each $j=1, \ldots, d$, $\lambda_j(\cdot)$ is constant on $\mathsf{sib}_F(u)$, and hence the equivalence of (i) and (ii) follows from \eqref{system-main}. The equivalence of (ii) and (iii) is immediate from the definition \eqref{def-Xj} of $X_j.$
\end{proof}
\begin{remark} \label{slam-E-dual}
It follows from the implication (i) $\Rightarrow$ (ii) above that 
$\mathcal L_{u, F}$ (and hence by \eqref{j-kernel-2} the joint kernel $E$ of $\slamc{c}^*$) is independent of the choice of $\mf c$. 
\end{remark}

The following result identifies the building blocks $\mathcal L_{u, F}$ appearing in the orthogonal decomposition of joint cokernel of $\slamc{c}$ with tensor product of kernel of $X_j$, $j \in F$.
\begin{theorem}  \label{dimE-thm}
Let $\mathscr T = (V,\mathcal E)$ be the directed Cartesian product of locally finite rooted directed trees $\mathscr T_1, \cdots, \mathscr T_d$ and
let $S_{\lambdab_{\mf c}}$ be a member of $\mathcal S_{\mathscr T}$. 
Let $\mathcal{L}_{u,F}$ be as appearing in \eqref{j-kernel-2} with $F=\{i_1, \ldots, i_k\}$ 
for some positive integer $k \in \{1, \ldots, d\}$. Then $\mathcal{L}_{u,F}$ is isomorphic to the finite dimensional space 
$\ker X_{i_1} \otimes \cdots \otimes \ker X_{i_k}$
where $X_j~(j \in F)$ is as defined in \eqref{def-Xj}.
In particular, \beq 
\label{dim-formula}
\dim \mathcal L_{u, F}=\prod_{j \in F} (\mbox{card}(\sib{u_{j}})-1). \eeq
\end{theorem}
\begin{proof}
We begin with the fact that $l^2(\sibi{F}{u})$ can be identified as the tensor product of $l^2(\sib{u_j})$, $j \in F.$ Since our proof utilizes the precise form of the isomorphism between these spaces, we provide elementary details essential in this identification.

Define $\phi :  l^2(\mathsf{sib}(u_{i_1})) \oplus  \cdots \oplus  l^2(\mathsf{sib}(u_{i_k})) \rar l^2(\sibi{F}{u})$ by
\beq \label{def-phi}
\phi((f_{i_1}, \ldots, f_{i_k}))(v)  = \prod_{j \in F} f_j(v_j), \quad \mbox{$v =(v_1, \ldots, v_d) \in \sibi{F}{u}$.} 
\eeq
It is easy to see that $\phi$ is multilinear. By the universal property of tensor product of vector spaces \cite[Theorem 4.14]{H}, there exists a unique linear map $\varPhi$ such that the following diagram commutes:
\beqn \label{uni-prop-1}
\begin{tikzcd}[row sep=15ex, column sep=10ex]
l^2(\mathsf{sib}(u_{i_1})) \oplus  \cdots \oplus  l^2(\mathsf{sib}(u_{i_k}) 
\arrow{rd}{\phi} \arrow{r}{\otimes}  & l^2(\mathsf{sib}(u_{i_1})) \otimes  \cdots \otimes  l^2(\mathsf{sib}(u_{i_k})  \arrow{d}{\varPhi}\\
& l^2(\sibi{F}{u})
\end{tikzcd}
\eeqn
\noindent
By \eqref{def-phi} and $\varPhi \circ \otimes=\phi$, the action of $\varPhi$ on elementary tensors is given by
\beq \label{phi-tilde}
\varPhi (f_{i_1} \otimes \cdots \otimes f_{i_k})(v_{1}, \ldots, v_{d}) = \prod_{j \in F} f_{j}(v_{j}), \quad v \in \sibi{F}{u}.
\eeq
The map $\varPhi$ turns out to be an isomorphism. Since we are not aware of an appropriate reference, we include necessary details. 
We first verify that $\varPhi$ is injective. Let $f= \sum_{j=1}^N f_{j1} \otimes \cdots \otimes f_{jk} \in \ker \varPhi.$  Suppose to the contrary that $f \neq 0$. By \cite[Lemma 1.1]{LC},   \beq \label{lin-ind} \mbox{$\{f_{ji} : j=1, \ldots, N\}$ is linearly independent for every integer $i=1, \ldots, k.$} \eeq Since $\varPhi(f)=0$, it follows that
\beqn
\sum_{j=1}^N f_{j1}(v_{i_1}) \cdots f_{jk}(v_{i_k})=0, \quad v \in \sibi{F}{u}.
\eeqn
Fixing all coordinates of $v \in \sibi{F}{u}$ except ${i^{\tiny \mbox{th}}_k}$, and using \eqref{lin-ind}, we conclude that
\beqn
f_{j1}(v_{i_1}) \cdots f_{jk-1}(v_{i_{k-1}})=0, \quad v \in \sibi{F}{u}, ~j=1, \ldots, N.
\eeqn
Since $f_{jk-1} \neq 0$, by fixing all coordinates of $v \in \sibi{F}{u}$ except ${i_{k-1}}^{\tiny th}$, we conclude that 
\beqn
f_{j1}(v_{i_1}) \cdots f_{jk-2}(v_{i_{k-2}})=0, \quad v \in \sibi{F}{u}, ~j=1, \ldots, N.
\eeqn
Continuing like this, we arrive at the conclusion that $f_{j1}$ is identically $0$ for $j=1, \ldots, N$, which contradicts \eqref{lin-ind}.
Hence we must have $f=0,$ that is, $\varPhi$ is injective.
Further, since $\mbox{card}(\sibi{F}{u})=\prod_{j \in F} \mbox{card}(\sib{u_{j}})$ (see \eqref{MuF-NuF}), we obtain
\beqn
\dim\big(l^2(\mathsf{sib}(u_{i_1})) \otimes  \cdots \otimes  l^2(\mathsf{sib}(u_{i_k})\big) =
\dim l^2(\sibi{F}{u}). 
\eeqn
It follows that  $\varPhi$ is an isomorphism.
However, for rest of the proof, we also need to know the action of $\varPhi^{-1}$. To see that, let $f  \in l^2(\sibi{F}{u}).$  Then 
\beqn
f = \sum_{ w \in \sibi{F}{u}} f(w) e_{w}
    = \sum_{w \in \sibi{F}{u}}  f(w) \prod_{i \in F}  \chi_{w_i}, 
    \eeqn     
    where, for $i \in F$ and $v \in \sibi{F}{u}$,  
\beqn  \chi_{w_i}(v) = \begin{cases}
     1 & \mbox{if}~ v_i=w_i,\\ 
    0 & \mbox{otherwise}.
    \end{cases}
    \eeqn
It is now easy to see using \eqref{phi-tilde} that
  \beq\label{phi-inv}
   \varPhi^{-1}(f) = \sum_{w \in \sibi{F}{u}}  f(w) (e_{w_{i_1}} \otimes \cdots \otimes e_{w_{i_k}}). 
  \eeq

We now check that $\varPhi$ maps $\ker X_{i_1} \otimes  \cdots \otimes \ker X_{i_k}$ into $\mathcal L_{u, F}.$ To see this, let $f_j \in l^2(\sib{u_{i_j}})$ be such that \beq \label{Xij=0} X_{i_j}f_j=0, \quad j=1, \ldots, k. \eeq 
In view of Lemma \ref{system-main-Sc-lem}, it suffices to check that 
\begin{align} \label{system-main-5}
\displaystyle \sum_{w \in \mathsf{sib}_{i_j}(v_G | u_{i_j})} \varPhi(f_1 \otimes \cdots \otimes f_k)(w)  = 0, \quad v_G \in \mathsf{sib}_{F, G}(u),~ G = F \setminus \{ i_j \}, ~j=1, \ldots, k.
\end{align}
However, for  $v_G \in \mathsf{sib}_{F, G}(u),~ G = F \setminus \{ i_j \}, ~ j =1, \ldots,  k$,
\beqn
 \sum_{w \in \mathsf{sib}_{i_j}(v_G | u_{i_j})} \varPhi(f_1 \otimes \cdots \otimes f_k)(w) &\overset{\eqref{phi-tilde}}=& \sum_{w \in \mathsf{sib}_{i_j}(v_G | u_{i_j})} \prod_{l=1}^k f_{l}(w_{i_l}) \\ &=& \sum_{\eta \in \mathsf{sib}(u_{i_j})} \Big(\prod_{\underset{l \neq j}{l=1}}^k f_{l}(v_{i_l})\Big)f_j(\eta) \\
 &\overset{\eqref{def-Xj}}=&\Big(\prod_{\underset{l \neq j}{l=1}}^k f_{l}(v_{i_l})\Big) X_{i_j}f_j,
\eeqn
which is $0$ in view of \eqref{Xij=0}. This yields \eqref{system-main-5}, and hence $$\varPhi\big(\ker X_{i_1} \otimes  \cdots \otimes \ker X_{i_k}\big) \subseteq\mathcal L_{u, F}.$$ To see that this inclusion is an equality, let $f \in \mathcal L_{u, F}$. By \cite[Lemma 4.1.5(i)]{CPT}, 
\beq \label{star}
\mathsf{sib}_F (u) = \displaystyle \bigsqcup_{v_G \in \mathsf{sib}_{F, G}(u)} \mathsf{sib}_j (v_G|u_j), \quad G=F \setminus \{j\}, ~j \in F. \eeq
It follows that for $G=F \setminus \{i_j\},$ $j=1, \ldots, k$,
\beqn
 \varPhi^{-1}(f) 
& \overset{\eqref{phi-inv}} =& \displaystyle  \sum_{w \in \sibi{F}{u}}  f(w)(e_{w_{i_1}} \otimes \cdots \otimes e_{w_{i_k}}) \\
&\overset{\eqref{star}}= & \displaystyle  \sum_{{v_G \in \mathsf{sib}_{F, G}(u)}} ~ \sum_{w \in \sibi{i_j}{v_G|u_{i_j}}}  f(w)(e_{w_{i_1}} \otimes \cdots \otimes e_{w_{i_k}})\\
&= &  \displaystyle \sum_{{\tiny v_G \in \mathsf{sib}_{F, G}(u)} }  e_{w_{i_1}} \otimes \cdots \otimes e_{w_{i_{j-1}}} \otimes { \sum_{\tiny w_{i_j} \in \mathsf{sib}(u_{i_j})} f(v_G|w_{i_j}) e_{w_{i_j}}} \otimes e_{w_{i_{j+1}}} \otimes \cdots \otimes e_{w_{i_k}}.
\eeqn
However, since 
$f \in \mathcal L_{u, F},$ by Lemma \ref{system-main-Sc-lem},
$$\sum_{\tiny w_{i_j} \in \mathsf{sib}(u_{i_j})} f(v_G|w_{i_j}) e_{w_{i_j}} \in \ker X_{i_j}.$$ It is now clear that $ \varPhi^{-1}(f) \in \ker X_{i_1} \otimes  \cdots \otimes \ker X_{i_k}.$
Since dimension of tensor product of vector spaces is product of dimensions of respective vector spaces \cite[Theorem 4.14]{H},
the remaining part is immediate.
\end{proof}

A careful examination of the proof of Theorem \ref{dimE-thm} shows that the formula for the joint cokernel holds for any multishift with constant weight system taking value $1$ (commonly known as {\it adjacency operator} in dimension $d=1$; refer to \cite{JJS}).
The following is immediate from \eqref{j-kernel-1}, \eqref{j-kernel-2}  and \eqref{dim-formula} (see also Lemma \ref{lem-E-finite}).
\begin{corollary} \label{dimE}
Let $\mathscr T = (V,\mathcal E)$ be the directed Cartesian product of locally finite rooted directed trees $\mathscr T_1, \cdots, \mathscr T_d$
of finite joint branching index. 
Let $S_{\lambdab_{\mf c}}$ be a member of $\mathcal S_{\mathscr T}$ and let $E$ denote the joint kernel of $S^*_{\lambdab_{\mf c}}$.  Then the dimension of $E$ is given by
\beqn
\dim E 
= \sum_{{F  \in \mathscr{P}}} \sum_{u \in \Omega_{F}} \prod_{i \in F} (\mbox{card}(\sib{u_{i}})-1) = 1+ \sum_{\underset{F \neq \emptyset}{F  \in \mathscr{P}}} \sum_{u \in \Omega_{F}} \prod_{i \in F} (\mbox{card}(\sib{u_{i}})-1).
\eeqn
\end{corollary}
\begin{remark} 
In case $d=1,$ the above formula for $E$ simplifies to
\beqn
\dim E = 1+  \sum_{u \in \Omega_{\{1\}}}  (\mbox{card}(\sib{u})-1),
\eeqn
which holds for arbitrary choice of positive weights $\lambdab$ $($see \cite[Proposition 3.5.1(ii)]{JJS}$)$. This formula resembles the expression for the (undirected) graph invariant $\Upsilon(\mathscr T)$ introduced in \cite[Pg 3]{AFFP}, which counts precisely the number of the so-called partial conjugations of the right angled Artin group $A_{\mathscr T}$ defined by deep nodes.
\end{remark}

\subsection{An analytic model}

In this section, we obtain an analytic model for multishifts belonging to the family $\mathcal S_{\mathscr T}$ (see \eqref{family}). 
The treatment here relies on a technique developed in 
the proof of \cite[Theorem 5.2.6]{CPT}.
We begin with 
an important aspect of the family $\mathcal S_{\mathscr T}$ that it is closed under the operation of taking spherical Cauchy dual.

\begin{lemma} \label{dual-class}
Let $\mathscr T = (V,\mathcal E)$ be the directed Cartesian product of locally finite rooted directed trees $\mathscr T_1, \cdots, \mathscr T_d$. 
If $S_{\lambdab_{\mf c}} \in \mathcal S_{\mathscr T}$, then
$S^{\mf s}_{\lambdab_{\mf c}}$ is well-defined and  $S^{\mf s}_{\lambdab_{\mf c}}$ belongs to $\mathcal S_{\mathscr T}$. 
\end{lemma}
\begin{proof}
Assume that $S_{\lambdab_{\mf c}}=(S_1, \ldots, S_d) \in \mathcal S_{\mathscr T}$.
Note that by \eqref{constant-gen}, $$\inf_{v \in V}\sum_{j=1}^d\|S_j e_v\|^2 = \inf {\mf c} > 0, \quad  \sup_{v \in V}\sum_{j=1}^d\|S_j e_v\|^2 = \sup {\mf c} < \infty.$$
It follows that $S_{\lambdab_{\mf c}}$ is joint left-invertible, and hence $S^{\mf s}_{\lambdab_{\mf c}}$ is well-defined. On the other hand, 
by \eqref{dual-wts} and \eqref{constant-gen}, the weights of $S^{\mf s}_{\lambdab_{\mf c}}$ are given by
\beqn
\lambda_{j}(w) = \frac{1}{\sqrt{\mf c(|\dep_v|)}}\sqrt{\frac{1}{{\mbox{card}(\childi{j}{v})}}} \sqrt{\frac{\dep_{v_j}  + 1}{|\dep_v| + d}}, \quad w \in \childi{j}{v},~v \in V, ~j=1, \ldots, d.
\eeqn
It is now clear that $S^{\mf s}_{\lambdab_{\mf c}} \in \mathcal S_{\mathscr T}$.
\end{proof}

We skip the proof of the following simple yet useful fact, which may be obtained by a routine inductive argument (cf. Proof of \cite[Corollary 5.2.12]{CPT}).
\begin{lemma} \label{effect-wts} 
Let $\mathscr T= (V, 
\mathcal{E})$ be the directed Cartesian product of 
locally finite rooted directed trees $\mathscr T_1, \ldots, \mathscr T_d$ and
let
$\slam$ be a commuting multishift on $\mathscr T$ with weight system $\lambdab=\{\lambda_j(v) : v \in V^{\circ}, ~ j=1, \ldots, d\}$. For a bounded sequence $\mf w$ of positive numbers, let $\lambdab_{\mf w}$ denote the system
\beq \label{p-wts}
 {\mf w}(|\dep_v|)\lambda_j(w), \quad w \in \childi{j}{v},~v \in V,~ j=1, \ldots, d.
\eeq
Then the multishift $S_{\lambdab_\mf w}$ on $\mathscr T$ with weight system as given in \eqref{p-wts} is commuting. 
Moreover, for any $v \in V$ and $\beta \in \mathbb N^d$, we obtain
\begin{enumerate}
\item[(i)]
$S^\beta_{\lambdab_\mf w} e_v = \Big(\prod_{p=0}^{|\beta|-1} {{\mf w({|\dep_v|+p})}}\Big) {\slam^{\beta}} e_v,$
\item[(ii)] $\|S^\beta_{\lambdab_\mf w} e_v\|^2 = \Big(\prod_{p=0}^{|\beta|-1} {{\mf w({|\dep_v|+p})^2}}\Big) \|{\slam^{\beta}} e_v\|^2$.
\end{enumerate}
\end{lemma}

Here is a key observation in obtaining an analytic model for members of $\mathcal S_{\mathscr T}$. 
\begin{lemma} \label{ortho-powers}
Let $\mathscr T = (V,\mathcal E)$ be the directed Cartesian product of locally finite rooted directed trees $\mathscr T_1, \ldots, \mathscr T_d$
of finite joint branching index. 
Let $S_{\lambdab_{\mf c}}$ be in the family $\mathcal S_{\mathscr T}$ and let $E$ denote the joint kernel of $S^*_{\lambdab_{\mf c}}$. Then the following statements are true:
\begin{enumerate}
\item[(i)] $E$ is invariant under $S^{*\alpha}_{\lambdab_{\mf c}} S^{\alpha}_{\lambdab_{\mf c}}$ and $S^{*\alpha}_{\lambdab_{\mf c}} S^{\alpha}_{\lambdab_{\mf c}}|_E$ is boundedly invertible for all $\alpha \in \mathbb N^d$. 
\item[(ii)] The multisequence $\{S^{\alpha}_{\lambdab_{\mf c}}E\}_{\alpha \in \mathbb N^d}$ of subspaces of $l^2(V)$ is mutually orthogonal.
\end{enumerate}
\end{lemma}
\begin{proof}
The proof relies on the case in which $\mf c = \mf c_a,$ $a >0$ (see 
\eqref{c-a}) as established in \cite[Lemma 5.2.7]{CPT}. 
Let $\alpha \in \mathbb N^d.$ We apply Lemma \ref{effect-wts}(ii) to the system $\lambdab_{\mf c_a}$ given by
\beq \label{weights-0}
\lambda_{j}(w) = \frac{1}{\sqrt{{\mbox{card}(\childi{j}{v})}}} \sqrt{\frac{\dep_{v_j}  + 1}{|\dep_v| + a}}, \quad w \in \childi{j}{v},~v \in V, ~j=1, \cdots, d,
\eeq 
with 
\beq
\label{m-perturb}
\mf w(t) := \sqrt{\mf c(t)} \sqrt{\frac{t+a}{t+d}}, \quad t \in \mathbb N,
\eeq
to conclude that \beq \label{2.21}
\|S^\alpha_{\lambdab_\mf c} e_v\|^2 = K(|\alpha|, |\dep_v|) \|{\slamc{c_a}^{\alpha}} e_v\|^2, \quad v \in V, \eeq where 
\beq \label{K(s, t)}
K(s, t)=\prod_{p=0}^{s-1} {{\mf w({t+p})^2}} ~\overset{\eqref{m-perturb}}=~\prod_{p=0}^{s-1} \mf c(t+p) \prod_{p=0}^{s-1}\frac{t+a+p}{t+d+p}, \quad s, t \in \mathbb N.\eeq
For $F \in \mathscr P$ and $u \in \Omega_F,$ let $f \in \mathcal L_{u, F}.$
It is now easy to see from Lemma \ref{ortho-v}(i) and \eqref{2.21} that
\beq
\label{s-s}
S^{*\alpha}_{\lambdab_\mf c} S^\alpha_{\lambdab_\mf c} f = K(|\alpha|, |\dep_u|) {\slamc{c_a}^{*\alpha}} {\slamc{c_a}^{\alpha}} f.
\eeq
Since $\mathcal L_{u, F}$ is invariant under $S^{*\alpha}_{\lambdab_{\mf c_a}} S^{\alpha}_{\lambdab_{\mf c_a}}$ (\cite[Proof of Lemma 5.2.7]{CPT}), by Remark \ref{slam-E-dual}, we must have $S^{*\alpha}_{\lambdab_{\mf c}} S^{\alpha}_{\lambdab_{\mf c}}f \in \mathcal L_{u, F} \subseteq E.$
Also, since $S^{*\alpha}_{\lambdab_{\mf c_a}} S^{\alpha}_{\lambdab_{\mf c_a}}|_E$ is boundedly invertible for all $\alpha \in \mathbb N^d$, 
the remaining part in (i) is now immediate from \eqref{s-s}, \eqref{K(s, t)}, Lemma \ref{lem-E-finite}(i) and the assumption that $\inf \mf c > 0$. 

To see (ii), let $\alpha \in \mathbb N^d$ and $$f = \displaystyle \sum_{{F \in \mathscr P}}\sum_{u \in \Omega_F} f_{u, F} \in E, \quad f_{u, F} \in \mathcal L_{u, F}.$$ 
Since $\mathcal L_{u, F} \subseteq l^2(\sibi{F}{u})$, we have \beq \label{LuF-g-1} \mathcal L_{u, F} \subseteq \bigvee \{e_v : v \in \mathcal G_{|\dep_u|}\}. \eeq
This combined with Lemma \ref{effect-wts}(i) and \eqref{K(s, t)} implies that $S^\alpha_{\lambdab_\mf c} f_{u, F}= \sqrt{K(|\alpha|, |\dep_u|)} {\slamc{c_a}^{\alpha}} f_{u, F}$.
Thus we obtain
\beqn
\slamc{c}^{\alpha}f = \sum_{{F \in \mathscr P}}\sum_{u \in \Omega_F} \slamc{c}^{\alpha} f_{u, F} = \sum_{{F \in \mathscr P}}\sum_{u \in \Omega_F} \sqrt{K(|\alpha|, |\dep_u|)} {\slamc{c_a}^{\alpha}} f_{u, F}.
\eeqn
The desired conclusion in (ii) now follows from the fact that $\{S^{\alpha}_{\lambdab_{\mf c_a}}E\}_{\alpha \in \mathbb N^d}$ is mutually orthogonal (see \cite[Lemma 5.2.7]{CPT}).
\end{proof}

We note that the conclusion of (ii) in the above lemma holds for any (spherically) balanced, injective weighted shift on a rooted directed tree (see \cite[Lemma 15 and Theorem 16]{BDPP}). We believe that this result fails in higher dimensions.

We now present the promised analytic model for multishifts belonging to the family $\mathcal S_{\mathscr T}$ (cf. \cite[Theorem 5.2.6]{CPT}, \cite[Theorem 2.2]{CT}). 
\begin{theorem}
\label{S-c-a-kernel}
Let $\mathscr T = (V,\mathcal E)$ be the directed Cartesian product of locally finite rooted directed trees $\mathscr T_1, \cdots, \mathscr T_d$ 
of finite joint branching index. 
Let $S_{\lambdab_{\mf c}}$ be a multishift belonging to $\mathcal S_{\mathscr T}$. Let $E$ denote the joint kernel of $S^*_{\lambdab_{\mf c}}$ and let $\mathcal{L}_{u,F}$ be as appearing in \eqref{j-kernel-2}.
Then $S_{\lambdab_{\mf c}}$ is unitarily equivalent to the multiplication $d$-tuple $\mathscr M_{z}=(\mathscr M_{z_1}, \cdots, \mathscr M_{z_d})$ on a reproducing kernel Hilbert space $\mathscr H_{\mf c}(\mathscr T)$ of $E$-valued holomorphic functions defined on the open ball $\mathbb B^d_r$ in $\mathbb C^d$ for some positive $r$. Further, the reproducing kernel $\kappa_{\mathscr H_{\mf c}(\mathscr T)} : \mathbb B^d_r \times \mathbb B^d_r \rar B(E)$ associated with $\mathscr H_{\mf c}(\mathscr T)$ is given by 
\beq \label{rk-formula}
 \kappa_{\mathscr H_{\mf c}(\mathscr T)}(z, w) =  \sum_{{F  \in \mathscr{P}}} \sum_{u \in \Omega_F} \Big (\sum_{\alpha \in \mathbb N^d}  \frac{\dep_{u}!}{(\dep_{u}+\alpha)!}\, {\prod_{j=0}^{|\alpha|-1}\frac{|\dep_u|+d + j}{\mf c(|\dep_u|+j)}}\,   z^{\alpha} \overline{w}^{\alpha} \Big) P_{\mathcal L_{u, F}}, ~ z, w \in \mathbb B^d_r, \quad \quad
\eeq
where $P_{\mathcal L_{u, F}}$ is the orthogonal projection on $\mathcal{L}_{u,F}$.
\end{theorem}
\begin{remark} \label{graph-iso}
Note that the reproducing kernel Hilbert space 
$\mathscr H_{\mf c}(\mathscr T)$ is a module over the polynomial ring $\mathbb C[z_1, \ldots, z_d]$ $($see \eqref{HM}$)$.
Fix $j=1,2$. Let $\mathscr T^{(j)}= (V^{(j)}, 
\mathcal{E}^{(j)})$ be the directed Cartesian product of 
locally finite rooted directed trees 
$\mathscr T^{(j)}_1, \ldots, \mathscr T^{(j)}_d$
of finite joint branching index. 
Let $S^{(j)}_{\lambdab_{\mf c}}$ be a member of $\mathcal S_{\mathscr T^{(j)}}$.
It may be concluded from \cite[Remark 3.1.1]{CPT} that if $\mathscr T^{(1)}$ and $\mathscr T^{(2)}$ are graph isomorphic, then the multishifts $\slamc{c}^{(1)}$ and $\slamc{c}^{(2)}$ are unitarily equivalent. It follows that the Hilbert modules $\mathscr H_{\mf c}(\mathscr T^{(1)})$ and $\mathscr H_{\mf c}(\mathscr T^{(2)})$ are isomorphic in this case.
\end{remark}
\begin{proof}
We adapt the argument of \cite[Theorem 5.2.6]{CPT} to the present situation.
The verification of the first part is along the lines of Step I of the proof of \cite[Theorem 5.2.6]{CPT}.  Indeed, the space $\mathscr H_{\mf c}(\mathscr T)$ can be explicitly written as   
 \beqn
\mathscr H_{\mf c}(\mathscr T) = \Big\{F(z)=\sum_{\alpha \in \mathbb N^d}f_{\alpha} z^{\alpha} : f_{\alpha} \in E~(\alpha \in \mathbb N^d),~\sum_{\alpha \in \mathbb N^d}\|S^{\alpha}_{\lambdab_{\mf c}} f_{\alpha}\|^2 < \infty \Big\}
\eeqn
with inner product 
\beqn \label{inner-product}
\inp{F(z)}{G(z)}_{\mathscr H_{\mf c}(\mathscr T)} = \sum_{\alpha \in \mathbb N^d}\inp{S^{\alpha}_{\lambdab_{\mf c}} f_{\alpha}}{S^{\alpha}_{\lambdab_{\mf c}} g_{\alpha}}_{l^2(V)}, \quad F, G \in \mathscr H_{\mf c}(\mathscr T).
\eeqn
We leave the details to the interested reader. We also skip the routine verification of the fact that $\kappa_{\mathscr H_{\mf c}(\mathscr T)}$ is a reproducing kernel for $\mathscr H_{\mf c}(\mathscr T)$:
\beqn
\inp{F}{\kappa_{\mathscr H_{\mf c}}(\cdot, w)g} =\inp{F(w)}{g}_E, \quad \mbox{$F \in \mathscr H_{\mf c}(\mathscr T)$, ~$g \in E,$ ~$w \in \mathbb B^d_r.$}
\eeqn

Let us check that the series on the right hand side of  \eqref{rk-formula}  converges for any $z, w$ in some open ball centred at the origin in $\mathbb C^d.$
Note that the reproducing kernel $\kappa_{\mathscr H_{\mf c}}$ admits the orthogonal decomposition: \beqn \kappa_{\mathscr H_{\mf c}(\mathscr T)}(z, w) = \bigoplus_{\underset{F  \in \mathscr{P}}{u \in \Omega_F}} \kappa_{u, F}(z, w)P_{\mathcal L_{u, F}},
 \eeqn
where, for $F \in \mathscr P$ and $u \in \Omega_F$, 
\beqn
\kappa_{u, F}(z, w)=\sum_{\alpha \in \mathbb N^d}  \frac{\dep_{u}!}{(\dep_{u}+\alpha)!}\, {\prod_{j=0}^{|\alpha|-1}\frac{|\dep_u|+d + j}{\mf c(|\dep_u|+j)}}\,   z^{\alpha} \overline{w}^{\alpha}. 
\eeqn
Further, the domain of convergence of $\kappa_{u,F}(\cdot, w)$ contains the open ball of radius ${\inf \mf c}$ for fixed $w$ in the unit ball in $\mathbb C^d$.
Indeed, $\inf \mf c$ is positive since $\slamc{c}$ belongs to $\mathcal S_{\mathscr T}$, and hence for any $F \in \mathscr P$ and $u \in \Omega_F$, 
 \beqn
|\kappa_{u, F}(z, w)| & \Le &  \sum_{\alpha \in \mathbb N^d}  \frac{\dep_{u}!}{(\dep_{u}+\alpha)!}\, {\prod_{j=0}^{|\alpha|-1}{(|\dep_u|+d + j)}\frac{1}{(\inf \mf c)^{|\alpha|}}}\,   |z^{\alpha}| |{w}^{\alpha}|  
\\
\\ & \overset{(\star)}\Le & \frac{\dep_u!}{(|\dep_u|+d-1)!}\frac{1}{(1-{(\inf \mf c)^{-1}}\inp{z}{w})^d}|z^{\beta}||w^{\gamma}|,
\eeqn
where the inequality $(\star)$ can be deduced from the multinomial formula (see \cite[Proof of Lemma 4.4]{GR} for details), and for $j=1, \ldots, d$,
\beqn
\beta_j =\begin{cases} -\dep_{u_j} & \mbox{if~}z_j \neq 0, \\
0 & \mbox{otherwise}, \quad \quad
\end{cases}
\gamma_j =\begin{cases} -\dep_{u_j} & \mbox{if~}w_j \neq 0, \\
0 & \mbox{otherwise}.
\end{cases}
\eeqn
Also, since $\mathscr T$ is of finite joint branching index, by Lemma \ref{lem-E-finite}(i),  for any $F \in \mathscr P$,  $\mathcal L_{u, F} \neq \{0\}$ for at most finitely many $u \in \Omega_F$.
It follows that for fixed $w \in \mathbb B^d$,
the domain of convergence of $\kappa_{\mathscr H_{\mf c}(\mathscr T)}(\cdot, w)$
contains the open ball of radius $\inf \mf c$. If $r:=\min \{\inf \mf c, 1\},$ then the absolute convergence of $\kappa_{\mathscr H_{\mf c}(\mathscr T)}(z, w)$ for $z, w \in \mathbb B^d_r$ is now immediate from $${\kappa_{\mathscr H_{\mf c}(\mathscr T)}(z, w)} = \overline{\kappa_{\mathscr H_{\mf c}(\mathscr T)}(w, z)}, \quad z, w \in \mathbb B^d_r.$$ 

To see \eqref{rk-formula}, let $f \in \mathcal L_{u, F}$ for $F \in \mathscr P$ and $u \in \Omega_F$.
One can argue as in Step II of the proof of \cite[Theorem 5.2.6]{CPT} to obtain
\beq
\label{rk-D-alpha}
\kappa_{\mathscr H_{\mf c}(\mathscr T)}(z, w) = \sum_{\alpha \in \mathbb N^d} D_{\alpha} z^{\alpha} \overline{w}^{\alpha}, \quad z, w \in \mathbb B^d_r,
\eeq
where $D_{\alpha}$ is the inverse of $S^{*\alpha}_{\lambdab_{\mf c}} S^{\alpha}_{\lambdab_{\mf c}}|_E$ as ensured by Lemma \ref{ortho-powers}(i). 
On the other hand, as recorded in Step III of the proof of \cite[Theorem 5.2.6]{CPT}, 
\beq
\label{moment-Sca}
\|S^{\alpha}_{\lambdab_{\mf c_a}} e_u\|^{-2} =  \frac{\dep_{u}!}{(\dep_{u}+\alpha)!} {(|\dep_u|+a)(|\dep_u|+a+1) \cdots (|\dep_u| + a + |\alpha|-1)}, \quad u \in V. 
\eeq
This combined with \eqref{2.21} and \eqref{K(s, t)}  yields that \beq \label{moment-gen} \|S^{\alpha}_{\lambdab_{\mf c}}e_u\|^{-2} = \frac{\dep_{u}!}{(\dep_{u}+\alpha)!}\, {\prod_{j=0}^{|\alpha|-1}\frac{|\dep_u|+d + j}{\mf c(|\dep_u|+j)}}, \quad u \in V.\eeq
The desired expression in \eqref{rk-formula} is now immediate from \eqref{LuF-g-1}, Lemma \ref{ortho-v}(i) and \eqref{rk-D-alpha}.
\end{proof}
\begin{remark} \label{int-s-k}
The reproducing kernel $\kappa_{\mathscr H_{\mf c}(\mathscr T)}(z, w)$ can be obtained by integrating a family of scalar-valued reproducing kernels (cf. \eqref{c-D-A-k}) with respect to a finite family of spectral measures. Indeed, 
\beqn
\kappa_{\mathscr H_{\mf c}(\mathscr T)}(z, w) =  \sum_{{F  \in \mathscr{P}}} \int_{\Omega_F} \Big (\sum_{\alpha \in \mathbb N^d}  \frac{\dep_{u}!}{(\dep_{u}+\alpha)!}\, {\prod_{j=0}^{|\alpha|-1}\frac{|\dep_u|+d + j}{\mf c(|\dep_u|+j)}}\,   z^{\alpha} \overline{w}^{\alpha} \Big)   dP_F(u), \quad z, w \in \mathbb B^d_r,
\eeqn
where $P_F(\cdot),$ $F \in \mathscr P$ is the spectral measure given by
\beqn
P_F(\sigma) = \sum_{u \in \sigma} P_{\mathcal L_{u, F}}, \quad \sigma \in \mathscr P(\Omega_F).
\eeqn
Further, note that Theorem \ref{S-c-a-kernel} fails in dimension $d \Ge 2$ if we relax the assumption that $\mathscr T$ is of finite joint branching index. Indeed, if $d = 2$, $\mathscr T_1$ is the binary tree (see \cite[Section 4.3]{JJS}) and $\mathscr T_2$ is the rooted directed tree $\mathscr T_{1, 0}$ without any branching vertex, then $\mathcal L_{u, \{1\}} \neq \{0\}$ for  infinitely many $u \in \Omega_{\{1\}}$ $($cf. Lemma \ref{lem-E-finite}(i)$)$, and hence it may be concluded from Lemma \ref{ortho-v} and \eqref{moment-Sca} that $S^{*\alpha}_{\lambdab_{\mf c_a}} S^{\alpha}_{\lambdab_{\mf c_a}}|_E$ is not boundedly invertible for $\alpha =(1, 1)$. 
 It is worth noting that this phenomenon is not possible in dimension $d=1$ (see \cite[Proposition 3.1]{CPT-1}).
\end{remark}

From now onwards, the pair $(\mathscr{M}_{z},\mathscr{H}_{\mf c}(\mathscr T))$, 
 as obtained in Theorem \ref{S-c-a-kernel}, will be referred to as the {\it analytic model} 
 of the multishift $S_{\lambdab_{\mf c}}$ on $\mathscr T$. 
In case $\mf c= \mf c_a$, 
$a >0$ (see \eqref{c-a}), the multishift $S_{\lambdab_{\mf c_a}}$ will be referred to as {\it Drury-Arveson-type multishift on $\mathscr T$} (see \eqref{weights-0}). 
In case each directed tree $\mathscr T_j$ is isomorphic to $\mathbb N,$ $\slamc{c_1}$ is unitarily equivalent to the Drury-Arveson $d$-shift,  $\slamc{c_d}$ is unitarily equivalent to the Szeg$\ddot{\mbox{o}}$ $d$-shift, while $\slamc{c_{\tiny d+1}}$ is unitarily equivalent to the Bergman $d$-shift (refer to \cite{GR} for elementary properties of classical Drury-Arveson-type multishifts).
The analytic model for $\slamc{c_a}$ can be described as follows.

\begin{corollary}
\label{S-c-a-kernel-0}
Let $\mathscr T = (V,\mathcal E)$ be the directed Cartesian product of locally finite rooted directed trees $\mathscr T_1, \cdots, \mathscr T_d$ 
of finite joint branching index and
let $\slamc{c_a}$ be the Drury-Arveson-type multishift on $\mathscr{T}$. Let $E$ denote the joint kernel of $\slamc{c_a}^*$ and let $\mathcal{L}_{u,F}$ be as appearing in \eqref{j-kernel-2}. Then $\slamc{c_a}$ is unitarily equivalent to the multiplication $d$-tuple $\mathscr M_{z}=(\mathscr M_{z_1}, \cdots, \mathscr M_{z_d})$ on a reproducing kernel Hilbert space $\mathscr H_{\mf c_a}(\mathscr T)$ of $E$-valued holomorphic functions defined on the open unit ball $\mathbb B^d$ in $\mathbb C^d$. Further, the reproducing kernel $\kappa_{\mathscr H_{\mf c_a}(\mathscr T)} : \mathbb B^d \times \mathbb B^d \rar B(E)$ associated with $\mathscr H_{\mf c_a}(\mathscr T)$ is given by 
\beqn
 \kappa_{\mathscr H_{\mf c_a}(\mathscr T)}(z, w) =  \sum_{{F  \in \mathscr{P}}} \sum_{u \in \Omega_F} \Big (\sum_{\alpha \in \mathbb N^d}  \frac{\dep_{u}!}{(\dep_{u}+\alpha)!}\, {\prod_{j=0}^{|\alpha|-1}(|\dep_u|+a + j)}\,   z^{\alpha} \overline{w}^{\alpha} \Big) P_{\mathcal L_{u, F}}, \quad z, w \in \mathbb B^d,
\eeqn
where $P_{\mathcal L_{u, F}}$ is the orthogonal projection on $\mathcal{L}_{u,F}$.
\end{corollary}

We discuss here one more instance in which Theorem \ref{S-c-a-kernel} is applicable. Let $\slamc{c} \in \mathcal S_{\mathscr T}$. By Lemma \ref{dual-class}, the spherical Cauchy dual $\slamc{c}^{\mf s}$ of $\slamc{c}$ belongs to $\mathcal S_{\mathscr T}$. By Theorem \ref{S-c-a-kernel}, $\slamc{c}^{\mf s}$ admits an analytic model, say, $(\mathscr M_z, \mathscr H^{\mf s}_{\mf c}(\mathscr T))$. 
Sometimes, we refer to the Hilbert module $\mathscr H^{\mf s}_{\mf c}(\mathscr T)$ as the {\it Cauchy dual of the Hilbert module} $\mathscr H_{\mf c}(\mathscr T)$. 
It follows that the spherical Cauchy dual $d$-tuple $S^{\mf s}_{\lambdab_{\mf c_a}}$ of the Drury-Arveson-type multishift $S_{\lambdab_{\mf c_a}}$ on $\mathscr T$ is unitarily equivalent to the multiplication $d$-tuple $(\mathscr M_{z_1}, \cdots, \mathscr M_{z_d})$ on the reproducing kernel Hilbert space $\mathscr H^{\mf s}_{\mf c_a}(\mathscr T)$ of $E$-valued holomorphic functions defined on the open unit ball $\mathbb B^d$ in $\mathbb C^d$. Further, the reproducing kernel $\kappa_{\mathscr H^{\mf s}_{\mf c_a}(\mathscr T)} : \mathbb B^d \times \mathbb B^d \rar B(E)$ associated with $\mathscr H^{\mf s}_{\mf c_a}(\mathscr T)$ is given by 
\beqn
 \kappa_{\mathscr H^{\mf s}_{\mf c_a}(\mathscr T)}(z, w) =  \sum_{{F  \in \mathscr{P}}} \sum_{u \in \Omega_F} \Big (\sum_{\alpha \in \mathbb N^d}  \frac{\dep_{u}!}{(\dep_{u}+\alpha)!}\, {\prod_{j=0}^{|\alpha|-1}\frac{(|\dep_u|+d + j)^2}{(|\dep_u|+a+j)}}\,   z^{\alpha} \overline{w}^{\alpha} \Big) P_{\mathcal L_{u, F}}, \quad z, w \in \mathbb B^d.
\eeqn
The last formula is immediate from
\beq
\label{moment-dual} 
\|(S^{\mf s }_{\lambdab_{\mf c_a}})^\alpha f\|^{-2} = \frac{\dep_{u}!}{(\dep_{u}+\alpha)!}\, {\prod_{j=0}^{|\alpha|-1}\frac{(|\dep_u|+d + j)^2}{(|\dep_u|+a+j)}}~\|f\|^2, \quad f \in \mathcal L_{u, F}, ~u \in V, ~F \in \mathscr P, \quad
\eeq
which, in turn, can be derived from \eqref{moment-gen} and the fact that 
\beq
\label{cd-form}
\mbox{$S^{\mf s }_{\lambdab_{\mf c_a}}$ is of the form $S_{\lambdab_{\mf c}}$  with $\mf c(t)=\frac{t+a}{t+d},$ $t \in \mathbb N$.}
\eeq
\begin{remark}
\label{rmk-constant-kernel}
The joint kernel of $S^*_{\lambdab_{\mf c_a}}$ is same as that of $S^{\mf s*}_{\lambdab_{\mf c_a}}$. Thus,
in the model spaces $\mathscr H_{\mf c_a}(\mathscr T)$ and $\mathscr H^{\mf s}_{\mf c_a}(\mathscr T)$ of $S_{\lambdab_{\mf c_a}}$ and $S^{\mf s}_{\lambdab_{\mf c_a}}$ respectively, the subspaces of constant functions are same. Indeed, they are equal to the joint kernel of  $S^*_{\lambdab_{\mf c_a}}$ (see Remark \ref{slam-E-dual}).
\end{remark}

\subsection{Operator-valued representing measures}

In this subsection, we formally introduce the notion of an operator-valued representing measure for the Hilbert module $\mathscr H_{\mf c}(\mathscr T)$.
This is reminiscent of the well-studied notion of the Berger measure appearing in the study of subnormal operators in one and several variables (refer to \cite{Co}, \cite{JL} and \cite{C-Y}). The main result here provides a necessary and sufficient condition to ensure its existence and uniqueness. We conclude this section by computing explicitly the representing measures for Drury-Arveson-type Hilbert modules $\mathscr H_{\mf c_a}(\mathscr T),$ $a \Ge d$ and their Cauchy dual Hilbert modules $\mathscr H^{\mf s}_{\mf c_a}(\mathscr T),$ $a < d$. 

\begin{definition} 
\label{def-ir}
Let $\mathscr T = (V,\mathcal E)$ be the directed Cartesian product of locally finite rooted directed trees $\mathscr T_1, \cdots, \mathscr T_d$
of finite joint branching index. 
Let $S_{\lambdab_{\mf c}}$ be a multishift belonging to $\mathcal S_{\mathscr T}$
and let $(\mathscr{M}_{z},\mathscr{H}_{\mf c}(\mathscr T))$
 be the analytic model
 of the multishift $S_{\lambdab_{\mf c}}$ on $\mathscr T$. 
Let $E$ denote the joint kernel of $S^*_{\lambdab_{\mf c}}$ as described in \eqref{j-kernel-2}. 
We say that $\mathscr H_{\mf c}(\mathscr T)$ {\it admits a representing measure} if there exists a $B(E)$-valued product measure $\rho_{{}_{\mathscr T}} \times \nu_{{}_{\mathscr T}}=(\rho_{{}_u} \times \nu_{_u})_{_{u \in \Omega_F, F \in \mathscr P}}$ supported on $[0,b] \times \partial\mathbb{B}^d,$ $b >0$ such that the following hold:
\begin{enumerate}
\item[(i)](Integral representation) For any $f \in E$ and $\alpha \in \mathbb N^d,$
$$\|z^{\alpha} f\|^2_{\mathscr H_{\mf c}(\mathscr T)} = \displaystyle
 \int_{[0, b]}  \int_{\partial\mathbb{B}^d}  s^{|\alpha |} |z^{\alpha}|^2 \inp{d\rho_{{}_{\mathscr T}}(s) \times d\nu_{{}_{\mathscr T}}(z)  f}{f},$$
\item[(ii)](Diagonal measure) $\rho_{{}_u}$ and $\nu_{_u}$ are scalar-valued measures such that for any $g_{u, F} \in \mathcal L_{u, F},$
\beq \label{prod-measure-def}
  \displaystyle d\rho_{{}_{\mathscr{T}}} (s) \times d\nu_{{}_{\mathscr{T}}}(z)\Big(\sum_{{F  \in \mathscr{P}}} \sum_{u \in \Omega_{F}}  g_{u,F} \Big)   = \displaystyle
  \sum_{{F  \in 
 \mathscr{P}}} \sum_{u \in \Omega_{F}} d\rho_{{}_u}(s) d\nu_{{}_u}(z)  g_{u,F}, 
 \eeq
\item[(iii)](Normalization) $\rho_{{}_{\rootb}}$ and $\nu_{_{\rootb}}$ are probability measures.
\end{enumerate}
\end{definition}

The existence of a representing measure is connected to the Hausdorff moment problem (refer to \cite[Chapter 4]{Si} for the definition and basic theory of Hausdorff moment sequences).
\begin{theorem}\label{slice-rep-gen} 
Let $\mathscr T = (V,\mathcal E)$ be the directed Cartesian product of locally finite rooted directed trees $\mathscr T_1, \cdots, \mathscr T_d$ of finite joint branching index. 
Let $S_{\lambdab_{\mf c}}$ be a multishift belonging to $\mathcal S_{\mathscr T}$ and 
let 
$(\mathscr{M}_{z},\mathscr{H}_{\mf c}(\mathscr T))$ 
 be the analytic model of
the multishift $S_{\lambdab_{\mf c}}$ on $\mathscr T$.
Then the following statements are equivalent:
\begin{enumerate}
\item[(i)] $\mathscr H_{\mf c}(\mathscr T)$ admits a representing measure $\rho_{{}_{\mathscr T}} \times \nu_{{}_{\mathscr T}}=(\rho_{{}_u} \times \nu_{_u})_{_{u \in \Omega_F, F \in \mathscr P}}$ supported on $[0, \sup \mf c] \times \partial\mathbb{B}^d$.
\item[(ii)] The sequence $\{\mf a_n\}_{n \in \mathbb N}$ given below is a Hausdorff moment sequence: \beq \label{seq-an} \mf a_n = \begin{cases} 1 & \mbox{if~} n=0, \\ 
\prod_{j=0}^{n-1} \mf c(j) & \mbox{if~} n \Ge 1.
\end{cases}
\eeq
\end{enumerate}
If any of the above statements holds, then for $u \in \Omega_F$ and $F \in \mathscr P$,
the positive scalar-valued measures $\rho_{{}_u}$ and $\nu_{_u}$ 
are given by 
\beq \label{drho}
d\rho_u(s) &=& \frac{s^{|\dep_u|}}{\mf a_{|\dep_u|}} d\rho_{\rootb}(s), \\ \label{dnu} d\nu_{{}_u}(z) &=& \frac{|z^{\dep_u}|^2 }{\|z^{\dep_u}\|_{L^2(\partial \mathbb B^d, \sigma)}^2}  d\sigma(z), \eeq
where $\rho_{\rootb}$ is the representing measure of $\{\mf a_n\}_{n \in \mathbb N}$ supported on $[0, \sup \mf c]$ and $\sigma$ is the normalized surface area measure on $\partial\mathbb{B}^d$.
\end{theorem}
\begin{remark}
The measure $\rho_{{}_{\rootb}}$, as appearing in \eqref{drho}, turns out to be the Berger measure of the weighted shift $S_{\theta}$ on $\mathscr T^{\otimes}_{\rootb}$ associated with $\slamc{c}$, where $\mathscr T^{\otimes}_{\rootb}$ is the connected component of the tensor product $\mathscr T^{\otimes}$ of $\mathscr T_1, \ldots, \mathscr T_d$ that contains $\rootb$ (see \cite[Definitions 2.2.1 and 5.2.13]{CPT} for definitions of $\mathscr T^{\otimes}$ and $S_{\theta}$). 
This may be deduced from \cite[(5.26)]{CPT}.
The representing measure of $\mathscr H_{\mf c}(\mathscr T)$ can be considered as complex and operator-valued version of the Berger measure arising in the subnormality of commuting tuples (refer to \cite{JL}, \cite{C-Y}).
\end{remark}
\begin{proof}
It follows from Lemma \ref{ortho-v}, \eqref{LuF-g-1} and \eqref{moment-gen} that for $f \in \mathcal L_{u, F}$ and $\alpha \in \mathbb N^d,$
\beq \label{moment-f}
 \|S^{\alpha}_{\lambdab_{\mf c}}f\|^2_{l^2(V)}  &=& \notag
   \frac{(\alpha+\dep_u)!}{\dep_u !}\prod_{j=0}^{|\alpha| -1}\frac{\mf c(|\dep_u|+j)}{|\dep_u|+d+j} \|f\|^2_{l^2(V)}  \\ &\overset{\eqref{seq-an}}=&   \frac{\mf a_{|\alpha|+|\dep_u|}}{\mf a_{|\dep_u|}}\, Q(\dep_u, \alpha)\, \|f\|^2_{l^2(V)},
  \eeq
where
$$Q(\beta, \alpha) :=  \frac{(\alpha+\beta)!}{\beta !}\prod_{j=0}^{|\alpha| -1} \frac{1}{|\beta|+d+j}, \quad \alpha, \beta \in \mathbb N^d.$$ 
By \cite[Lemma 1.11]{Z},  for $\alpha, \beta \in \mathbb{N}^d$, we obtain 
\beqn \label{spherical}
\int_{\partial\mathbb{B}^d}  \frac{|z^{\alpha+ \beta}|^2}{\|z^{\beta}\|_{L^2(\partial \mathbb B^d, \sigma)}^2}  d\sigma(z)
= \frac{\|z^{\alpha+\beta}\|^2_{L^2(\partial \mathbb B^d, \sigma)}}{\|z^{\beta}\|^2_{L^2(\partial \mathbb B^d, \sigma)}}  =
 \frac{ (\alpha+\beta)! (d-1)!}{(|\beta|+|\alpha|+d-1)!}  \frac{(|\beta| + d -1 )!}{\beta !(d-1)!}
= Q(\beta, \alpha). 
\eeqn
It is now immediate from \eqref{moment-f} that 
\beq \label{moment-f-gen}
 \|S^{\alpha}_{\lambdab_{\mf c}}f\|^2_{l^2(V)} =
\frac{\mf a_{|\alpha|+|\dep_u|}}{\mf a_{|\dep_u|}}\Big(\int_{\partial\mathbb{B}^d}  \frac{|z^{\alpha+\dep_u}|^2}{\|z^{\dep_u}\|_{L^2(\partial \mathbb B^d, \sigma)}^2}  d\sigma(z)\Big) \|f\|^2_{l^2(V)}, \quad f \in \mathcal L_{u, F}.
\eeq

{\it (ii) $\Rightarrow$ (i)}:
Assume that (ii) holds. Thus there exists a probability measure $\mu_{\mf c}$ supported on a finite interval $[0, {b}]$ such that
\beq \label{an-moment}
\mf a_n = \int_{[0, b]} s^{n} d\mu_{\mf c}(s), ~n \in \mathbb N. 
\eeq
By \cite[Lemma 2]{BJJS} and \eqref{seq-an},  $b = \sup \mf c$.
It is easy to see from  \eqref{moment-f-gen} and \eqref{an-moment}  that for $f \in \mathcal L_{u, F}$,
\beq\label{slice-norm}
\|z^\alpha f\|^2_{\mathscr H_{\mf c}(\mathscr T)} = \|S^{\alpha}_{\lambdab_{\mf c}}f\|^2_{l^2(V)}   
= 
\Big(\int_{[0, b]}\int_{\partial\mathbb{B}^d}  s^{|\alpha |}   |z^{\alpha}|^2 {d\rho_{{}_u}(s) d\nu_{{}_u}(z) }\Big){\|f\|^2_{l^2(V)}}, \eeq
where $d\rho_{{}_u}$ and $d\nu_{{}_u}$ are as defined in \eqref{drho} (with $\rho_{\rootb}$ replaced by $\mu_{\mf c}$) and \eqref{dnu} respectively.  
 To see the integral representation of $\|z^\alpha f\|^2_{\mathscr H_{\mf c}(\mathscr T)}$ for arbitrary $f \in E,$   
note that by \eqref{j-kernel-2}, any $f \in  E$ is of the form  
\beqn
f= \sum_{{F  \in \mathscr{P}}} \sum_{u \in \Omega_{F}}  g_{u,F},  \quad g_{u,F} \in \mathcal{L}_{u,F}.
\eeqn 
By \eqref{prod-measure-def},
\beqn
\int_{[0, b]}  \int_{\partial\mathbb{B}^d}  s^{|\alpha |} |z^{\alpha}|^2 \inp{d\rho_{{}_{\mathscr T}} \times d\nu_{{}_{\mathscr T}}  f}{f}    &=& \sum_{{F  \in 
 \mathscr{P}}} \sum_{u \in \Omega_{F}} \Big(\int_{[0, b]}   \int_{\partial\mathbb{B}^d}  s^{|\alpha|} |z^{\alpha}|^2  d\rho_{{}_u}(s) d\nu_{{}_u}(z) \Big) \| g_{u,F}\|^2    \\
 &\overset{\eqref{slice-norm}}=&  \sum_{{F  \in 
 \mathscr{P}}} \sum_{u \in \Omega_{F}} \| z^{\alpha}g_{u,F}\|^2_{\mathscr H_{\mf c}(\mathscr T)} \\ &=& \| z^{\alpha}f\|^2_{\mathscr H_{\mf c}(\mathscr T)},
\eeqn
where we used the orthogonality of $\{z^{\alpha}g_{u, F} : F \in \mathscr P, u \in \Omega_F\}$ in the last equality (see Lemma \ref{ortho-v}).
This completes the proof of (ii) $\Rightarrow$ (i).

{\it (i) $\Rightarrow$ (ii)}: Assume that (i) holds. 
It may be concluded from \cite[(5.24) and (5.26)]{CPT} that 
\beq \label{sp-gen}
\sum_{\underset{|\alpha|=n}{\alpha \in \mathbb N^d}} \frac{n!}{\alpha!} \|S^{\alpha}_{\lambdab_{\mf c}} e_{v}\|^2_{l^2(V)}=\prod_{p=0}^{n-1} \mf c(|\dep_v| + p), \quad n \in \mathbb N,~
v \in V. 
\eeq
Letting $f= e_{\rootb}$ in the integral representation (see Definition \ref{def-ir}(i)), we obtain
for any $n \in \mathbb N,$ \beqn
\mf a_n &=& \prod_{p=0}^{n-1} \mf c(p)  \overset{\eqref{sp-gen}}= \sum_{\underset{|\alpha|=n}{\alpha \in \mathbb N^d}} \frac{n!}{\alpha!} \|z^{\alpha} e_{\rootb}\|^2_{\mathscr H_{\mf c}(\mathscr T)} \\ &=& 
\Big( \int_{[0, b]}  s^{n} d\rho_{{}_{\rootb}}(s) \Big) \int_{\partial\mathbb{B}^d} \sum_{\underset{|\alpha|=n}{\alpha \in \mathbb N^d}} \frac{n!}{\alpha!}  |z^{\alpha}|^2d\nu_{{}_{\rootb}}(z) \\ &=& \int_{[0, b]}  s^{n} d\rho_{{}_{\rootb}}(s), 
\eeqn
where we used the assumption that $\nu_{{}_{\rootb}}$ is a probability measure along with the multinomial theorem in the last equality. This completes the verification of (i) $\Rightarrow$ (ii). 

To see the uniqueness part, note that by \eqref{seq-an} and  \eqref{sp-gen}, the sequence $\{\mf a_n\}_{n \in \mathbb N}$ is uniquely determined by the action of $\slamc{c}$ on $e_{\rootb}$.
By the determinacy of the Hausdorff moment problem \cite[Theorem 4.17.1]{Si}, the probability measure $\rho_{\rootb}$ is unique. It now follows from \eqref{drho} and \eqref{dnu} that the representing measure $\rho_{{}_{\mathscr T}} \times \nu_{{}_{\mathscr T}}$ of $\mathscr H_{\mf c}(\mathscr T)$ is unique.
\end{proof}
 
Let us see two particular instances in which representing measures can be determined explicitly.
 
 \begin{corollary}\label{coro-1-slice-rep} 
Let $\mathscr T = (V,\mathcal E)$ be the directed Cartesian product of locally finite rooted directed trees $\mathscr T_1, \cdots, \mathscr T_d$ 
of finite joint branching index and
let 
$(\mathscr{M}_{z}, \mathscr{H}_{\mf c_a}(\mathscr T))$ 
 be the {analytic model} of
the Drury-Arveson-type multishift $\slamc{c_a}$ on $\mathscr T$. If $a$ is a positive integer such that $a \Ge d$,
then $\mathscr{H}_{\mf c_a}(\mathscr T)$ admits the representing measure $\rho_{{}_{\mathscr{T}}}  \times \nu_{{}_{\mathscr{T}}}=(\rho_{{}_u} \times \nu_{_u})_{_{u \in \Omega_F, F \in \mathscr P}}.$
In this case, for $u \in \Omega_F$ and $F \in \mathscr P,$ the positive scalar-valued measures $\rho_{{}_u}$ and $\nu_{_u}$ are given by 
\beq \label{drho-dnu}
d\rho_{{}_u}(s)  = \begin{cases} w_{|\dep_u |}(s) d{m}(s) & \mbox{if~}a > d, \\
d\delta_1(s) & \mbox{if~} a=d, \end{cases} \quad d\nu_{{}_u}(z)=\frac{|z^{\dep_u}|^2 }{\|z^{\dep_u}\|_{L^2(\partial \mathbb B^d, \sigma)}^2}  d\sigma(z), \eeq
where ${m}$ is the Lebesgue measure on $[0, 1]$, $\delta_1$ is the Borel probability measure supported at $\{1\}$, $\sigma$ is the normalized surface area measure on $\partial\mathbb{B}^d$, and
\beq \label{wt-fn} w_l(s) =  {(l+d)\cdots(l+a-1)} \displaystyle \sum_{i=d}^{a-1}\frac{s^{i + l-1}}{\displaystyle
\prod_{d \Le j \neq i \Le a-1}(j -i)}, \quad s \in [0,1], ~ l \in \mathbb N.
\eeq
\end{corollary}
\begin{proof} 
Suppose that $a$ is a positive integer such that $a \Ge d$.
By \eqref{c-a}, $\mf c_a(t) = \frac{t+d}{t+a},~t \in \mathbb N$, and hence by \eqref{seq-an},
\beq  \label{proof-an} \mf a_n = \begin{cases} 1 & \mbox{if~} n=0, \\ 
\prod_{j=0}^{n-1} \frac{j+d}{j+a} & \mbox{if~} n \Ge 1.
\end{cases}
\eeq
Note that in case $a=d$, $\mf a$ is the constant sequence with value $1.$ Clearly, it is a Hausdorff moment sequence with representing measure $\delta_1.$ Suppose that $a > d.$ Then 
\beqn  \mf a_n = 
\prod_{j=d}^{a-1} \frac{j}{n+j}, \quad n \in \mathbb N.
\eeqn
It follows from \cite[Corollary 3.8]{AC} that $\{\mf a_n\}_{n \in \mathbb N}$ is a Hausdorff moment sequence with representing measure $\rho_{{}_{\rootb}}$
being the weighted Lebesgue measure with weight
function $w_0 : [0, 1] \rar [0, \infty)$ as given by \eqref{wt-fn}.
Thus $$d\rho_{{}_{\rootb}}(s)=\begin{cases} d\delta_1 & \mbox{if~}a=d, \\ 
w_0(s)dm(s) & \mbox{if~}a > d. \end{cases}$$ It follows that
\beqn
d\rho_u(s) ~\overset{\eqref{drho}} =~ \frac{s^{|\dep_u|}}{\mf a_{|\dep_u|}} d\rho_{{}_{\rootb}}(s) 
~=~ \begin{cases}  \frac{s^{|\dep_u|}}{\mf a_{|\dep_u|}} d\delta_1 \overset{\eqref{proof-an}}= d\delta_1  & \mbox{if~}a=d, \\
\frac{s^{|\dep_u|}}{\mf a_{|\dep_u|}}w_0(s)dm(s) \overset{\eqref{wt-fn}}= ~w_{|\dep_u|}(s)dm(s) & \mbox{if~} a > d.
\end{cases}
\eeqn
The expression for $d\nu_{_u}$ in \eqref{drho-dnu} follows from \eqref{dnu}.
\end{proof}

\begin{corollary}\label{coro-2-slice-rep} 
Let $\mathscr T = (V,\mathcal E)$ be the directed Cartesian product of locally finite rooted directed trees $\mathscr T_1, \cdots, \mathscr T_d$ 
of finite joint branching index
and let 
$(\mathscr{M}_{z}, \mathscr{H}^{\mf s}_{\mf c_a}(\mathscr T))$ 
 be the analytic model of
the spherical Cauchy dual $\slamc{c_a}^{\mf s}$ of the Drury-Arveson-type multishift $\slamc{c_a}$ on $\mathscr T$.
If $a$ is a positive integer such that $a < d$, then $\mathscr{H}^{\mf s}_{\mf c_a}(\mathscr T)$ admits the representing measure $\rho_{{}_{\mathscr{T}}}  \times \nu_{{}_{\mathscr{T}}}=(\rho_{{}_u} \times \nu_{_u})_{_{u \in \Omega_F, F \in \mathscr P}}.$
In this case, for $u \in \Omega_F$ and $F \in \mathscr P,$ the positive scalar-valued measures $\rho_{{}_u}$ and $\nu_{_u}$ are given by 
\beq \label{drho-dnu-2}
d\rho_{{}_u}(s)  = \omega_{|\dep_u |}(s) d{m}(s), \quad d\nu_{{}_u}(z)=\frac{|z^{\dep_u}|^2 }{\|z^{\dep_u}\|_{L^2(\partial \mathbb B^d, \sigma)}^2}  d\sigma(z), \eeq
where ${m}$ is the Lebesgue measure on $[0, 1]$, $\sigma$ is the normalized surface area measure on $\partial\mathbb{B}^d$, and
\beq \label{wt-fn-2} \omega_{l}(s) =  {(l+a)\cdots(l+d-1)} \displaystyle \sum_{i=a}^{d-1}\frac{s^{i + l-1}}{\displaystyle
\prod_{a \Le j \neq i \Le d-1}(j -i)}, \quad s \in [0,1], ~ l \in \mathbb N.
\eeq
\end{corollary}
\begin{proof} 
Suppose that $a$ is a positive integer such that $a < d$.
By \eqref{cd-form} and \eqref{seq-an},
\beqn  \mf a_n = \begin{cases} 1 & \mbox{if~} n=0, \\ 
\displaystyle \prod_{j=0}^{n-1} \frac{j+a}{j+d} = \prod_{j=a}^{d-1} \frac{j}{n+j} & \mbox{if~} n \Ge 1.
\end{cases}
\eeqn
It follows from \cite[Corollary 3.8]{AC} that $\{\mf a_n\}_{n \in \mathbb N}$ is a Hausdorff moment sequence with representing measure 
being the weighted Lebesgue measure with weight
function $\omega_0 : [0, 1] \rar [0, \infty)$ as given by 
\eqref{wt-fn-2}.
The desired expressions in \eqref{drho-dnu-2} can now be derived as in the proof of Corollary \ref{coro-1-slice-rep}.
\end{proof}

\section{Proof of the main result}

In this section, we present a proof of Theorem \ref{main-thm}. We begin with a simple observation which reduces Problem \ref{prob} to the problem of unitary equivalence of representing measures arising from the Drury-Arveson-type Hilbert modules (as ensured by Corollaries \ref{coro-1-slice-rep} and \ref{coro-2-slice-rep}). 
\begin{lemma}\label{set-measure-lem}
Fix $j=1,2$. Let $\mathscr T^{(j)}= (V^{(j)}, 
\mathcal{E}^{(j)})$ be the directed Cartesian product of 
locally finite rooted directed trees 
$\mathscr T^{(j)}_1, \ldots, \mathscr T^{(j)}_d$ of finite joint branching index. 
Let $S^{(j)}_{\lambdab_{\mf c}}$ be a multishift belonging to $\mathcal S_{\mathscr T^{(j)}}$ and let $(\mathscr M^{(j)}_z, \mathscr H_{\mf c}(\mathscr T^{(j)}))$ be the analytic model of the multishift  $S^{(j)}_{\lambdab_{\mf c}}$ on $\mathscr T^{(j)}$. Let 
$E^{(j)}$ be the subspace of constant functions in $\mathscr H_{\mf c}(\mathscr T^{(j)})$.
If $\mathscr H_{\mf c}(\mathscr T^{(j)})$ admits the representing measure $\rho_{_{{\mathscr{T}^{(j)}}}}  \times \nu_{_{{\mathscr{T}^{(j)}}}}=(\rho^{(j)}_{_u} \times \nu^{(j)}_{_u})_{_{u \in \Omega_F, F \in \mathscr P}}$ supported on $[0, 1] \times \partial \mathbb B^d$, then
the Hilbert modules $\mathscr{H}_{\mf c}(\mathscr T^{(1)})$ and  $\mathscr{H}_{\mf c}(\mathscr T^{(2)})$ are isomorphic
if and only if there exists a unitary transformation $U : E^{(1)} \rar E^{(2)}$ such that 
\beq \label{4.1} \quad \quad \quad U \rho_{{}_{\mathscr{T}^{(1)}}} \times \nu_{{}_{\mathscr{T}^{(1)}}}(A) U^{*} =  \rho_{_{\mathscr{T}^{(2)}}} \times \nu_{_{\mathscr{T}^{(2)}}}(A) ~\mbox{for every Borel subset $A$ of $[0, 1] \times \partial \mathbb B^d$.} \quad \quad  \eeq 
\end{lemma}
\begin{proof}
Suppose that $\mathscr H_{\mf c}(\mathscr T^{(j)})$ admits the representing measure $\rho_{_{{\mathscr{T}^{(j)}}}}  \times \nu_{_{{\mathscr{T}^{(j)}}}}$ supported on $[0, 1] \times \partial \mathbb B^d$, $j=1,2.$
Let $U : \mathscr{H}_{\mf c}(\mathscr T^{(1)}) \rar  \mathscr{H}_{\mf c}(\mathscr T^{(2)})$ be a unitary map such that 
\beq\label{intw} 
U\mathscr{M}^{(1)}_{z_j} U^{*} = \mathscr{M}^{(2)}_{z_j}, \quad j=1, \ldots, d.
\eeq It follows that 
\beq \label{ker-to-ker}
\mbox{$U$ maps $E^{(1)}=\displaystyle \bigcap_{j=1}^d \ker \mathscr M^{(1)^*}_{z_j}$ unitarily onto $E^{(2)}=\displaystyle \bigcap_{j=1}^d  \ker \mathscr M^{(2)^*}_{z_j}.$}
\eeq 
Thus, if $f \in E^{(1)}$, then $Uf \in E^{(2)}$, and hence
\beqn
 \int_{0}^{1} \int_{\partial\mathbb{B}^d}  s^{|\alpha |}|z^{\alpha}|^2 \inp{ d\rho_{{}_{\mathscr{T}^{(1)}}}(s) \times d\nu_{{}_{\mathscr{T}^{(1)}}}(z) f}{f} &=&  \|z^{\alpha}f\|^2_{\mathscr{H}_{\mf c}(\mathscr T^{(1)})}  =  \|U (\mathscr{M}^{(1)}_{z,a})^{\alpha}f\|^2_{\mathscr{H}_{\mf c}(\mathscr T^{(2)})} 
                      \\ &=&  \|(\mathscr{M}^{(2)}_{z,a})^{\alpha}U f\|^2_{\mathscr{H}_{\mf c}(\mathscr T^{(2)})} 
                    =  \|z^{\alpha}U f\|^2_{\mathscr{H}_{\mf c}(\mathscr T^{(2)})} 
                      \\ &=& \int_{0}^{1} \int_{\partial\mathbb{B}^d}  s^{|\alpha |}|z^{\alpha}|^2 \inp{ d\rho_{{}_{\mathscr{T}^{(2)}}}(s) \times d\nu_{{}_{\mathscr{T}^{(2)}}}(z) Uf}{Uf} \\
                      &=&  \int_{0}^{1} \int_{\partial\mathbb{B}^d}  s^{|\alpha |}|z^{\alpha}|^2 \inp{ U^* d\rho_{{}_{\mathscr{T}^{(2)}}}(s) \times d\nu_{{}_{\mathscr{T}^{(2)}}}(z) Uf}{f}.
\eeqn 
By uniqueness of the representing measure (see Theorem \ref{slice-rep-gen}), we obtain the necessary part.

To see the converse, assume that \eqref{4.1} holds for a unitary transformation $U : E^{(1)} \rar E^{(2)}$.
Define $\tilde{U} : \mathscr{H}_{\mf c}(\mathscr T^{(1)}) \rar \mathscr{H}_{\mf c}(\mathscr T^{(2)})$ by \beq \label{unitary} (\tilde{U}f)(z)=U(f(z)), \quad f \in \mathscr{H}_{\mf c}(\mathscr T^{(1)}), ~z \in \mathbb B^d.\eeq 
It is easy to see using \eqref{4.1} that $\tilde{U}$ is a unitary map. Also, it is a routine matter to verify that $$\tilde{U}\mathscr M^{(1)}_{z_j} = \mathscr M^{(2)}_{z_j}\tilde{U}, \quad j=1, \ldots, d.$$ This completes the proof.
\end{proof}

The following rather technical result says that any Hilbert module isomorphism between two Drury-Arveson-type Hilbert modules preserves the orthogonal decomposition \eqref{j-kernel-2} of joint cokernels of associated multiplication tuples over each generation.
\begin{proposition}\label{level-eq} 
Let $a, d$ be positive integers such that $ad \neq 1,$ and fix $j=1,2$. Let $\mathscr T^{(j)}= (V^{(j)}, 
\mathcal{E}^{(j)})$ be the directed Cartesian product of 
locally finite rooted directed trees 
$\mathscr T^{(j)}_1, \ldots, \mathscr T^{(j)}_d$ of finite joint branching index. 
Let $\rootb^{(j)}$ denote the root of $\mathscr T^{(j)}$ and
let $\mathscr H_{\mf c_a}(\mathscr T^{(j)})$ be the Drury-Arveson-type Hilbert module associated with $\mathscr T^{(j)}$. Let 
$E^{(j)}$ be the subspace of constant functions in $\mathscr H_{\mf c_a}(\mathscr T^{(j)})$
and let $\mathcal{L}^{(j)}_{u,F}$ be as appearing in the decomposition \eqref{j-kernel-2} of $E^{(j)}$.  Suppose there exists a Hilbert module isomorphism $U : \mathscr{H}_{\mf c_a}(\mathscr T^{(1)}) \rar  \mathscr{H}_{\mf c_a}(\mathscr T^{(2)})$.  Then, for any $\alpha \in \mathbb N^d$ and $F \in \mathscr P$, 
\beqn U~\mbox{ maps~} \displaystyle \bigoplus_{\underset{\dep_u=\alpha}{u \in \Omega_{F}^{(1)}}} \mathcal{L}^{(1)}_{u,F}~ \mbox{onto~} \displaystyle \bigoplus_{\underset{\dep_v=\alpha}{v\in \Omega_{F}^{(2)}}} \mathcal{L}^{(2)}_{v,F}, \eeqn
where we used the convention that orthogonal direct sum over empty collection is $\{0\}.$  
In particular, $U$ maps $e_{\rootb^{(1)}}$ to a unimodular scalar multiple of $e_{\rootb^{(2)}}$.  
\end{proposition}
\begin{proof}
Note that two joint left-invertible tuples are unitarily equivalent if and only if their spherical Cauchy dual $d$-tuples are unitarily equivalent. 
It follows that the Hilbert modules $\mathscr H_{\mf c}(\mathscr T^{(1)})$ and $\mathscr H_{\mf c}(\mathscr T^{(2)})$ are isomorphic if and only if their Cauchy dual Hilbert modules $\mathscr H^{\mf s}_{\mf c}(\mathscr T^{(1)})$ and $\mathscr H^{\mf s}_{\mf c}(\mathscr T^{(2)})$ are isomorphic.
Since $\mathscr H_{\mf c_a}(\mathscr T)~(a \Ge d)$ and $\mathscr H^{\mf s}_{\mf c_a}(\mathscr T)~(a < d)$ admit representing measures (Corollaries \ref{coro-1-slice-rep}  and \ref{coro-2-slice-rep}), in view of Remark \ref{slam-E-dual}, it suffices to treat the case in which $a \Ge d$.

Suppose that $a \Ge d.$ Then, by Corollary \ref{coro-1-slice-rep}, $\mathscr H_{\mf c_a}(\mathscr T^{(j)})$ admits the representing measure $\rho_{_{{\mathscr{T}^{(j)}}}}  \times \nu_{_{{\mathscr{T}^{(j)}}}}=(\rho^{(j)}_{_u} \times \nu^{(j)}_{_u})_{_{u \in \Omega_F, F \in \mathscr P}}$ as given by \eqref{drho-dnu}.
Let $f \in \mathcal{L}^{(1)}_{u,F}$ for some $u\in \Omega^{(1)}_{F}$ and 
 $F \in \mathscr P$.  
Then, by \eqref{ker-to-ker}, $U$ maps $E^{(1)}$ into $E^{(2)}$, and hence
 \beq
 \label{Uf}
Uf= \sum_{F  \in \mathscr{P}} \sum_{v \in \Omega_{F}^{(2)}}  g_{v,F},  \quad g_{v,F} \in \mathcal{L}_{v,F}^{(2)}.
\eeq 
Further, by Lemma \ref{set-measure-lem}, for any Borel subset $\Delta_1 \times \Delta_2$ of $[0,1] \times \partial \mathbb{B}^d,$
\beq \label{intertwining}
\int_{\Delta_1} \int_{\Delta_2} d\rho_{{}_{\mathscr{T}^{(2)}}}(s) \times d\nu_{{}_{\mathscr{T}^{(2)}}} (z)Uf &=& \int_{\Delta_1} \int_{\Delta_2} U d\rho_{{}_{\mathscr{T}^{(1)}}}(s) \times d\nu_{{}_{\mathscr{T}^{(1)}}}(z) f. \eeq
We verify that 
for almost every $z \in \partial \mathbb B^d$,
 \beq \label{depth-u-v-same}
  \frac{|z^{\dep_v}|^2}{\|z^{\dep_v}\|^{2}_{L^2(\partial \mathbb B^d, \sigma)}}  = \frac{|z^{\dep_u}|^2}{\|z^{\dep_u}\|^{2}_{L^2(\partial \mathbb B^d, \sigma)}} \quad\mbox{if~}v \in \Omega_{F}^{(2)}, ~g_{v, F} \neq 0.
 \eeq
 We divide the verification into two cases: 

\begin{case1*} When $a=d$.
\end{case1*}
\noindent
By the definition of the representing measure (see \eqref{prod-measure-def}) and Corollary \ref{coro-1-slice-rep}, 
 \beq \label{action-f}
 \int_{\Delta_1} \int_{\Delta_2}  U d\rho_{{}_{\mathscr{T}^{(1)}}}(s) \times d\nu_{{}_{\mathscr{T}^{(1)}}}(z) f &=& \int_{\Delta_1} \int_{\Delta_2}   \frac{|z^{\dep_{u}}|^2}{\|z^{\dep_{u}}\|_{L^2(\partial \mathbb B^d, \sigma)}^2}  d\delta_1(s) d\sigma (z)~Uf. 
 \eeq
Arguing similarly and using \eqref{Uf}, we obtain
\beq \label{action-Uf}
\int_{\Delta_1} \int_{\Delta_2} d\rho_{{}_{\mathscr{T}^{(2)}}}(s) \times d\nu_{{}_{\mathscr{T}^{(2)}}} (z)Uf = \int_{\Delta_1} \int_{\Delta_2}  \Big( 
   \sum_{F  \in 
 \mathscr{P}} \sum_{v \in \Omega_{F}^{(2)}}  \frac{|z^{\dep_v}|^2}{\|z^{\dep_v}\|^{2}_{L^2(\partial \mathbb B^d, \sigma)}}  g_{v,F}  \Big) d\delta_1(s) d\sigma(z). \quad \quad 
\eeq
Hence, by \eqref{action-Uf}, \eqref{intertwining} and \eqref{action-f}, we obtain 
\beq \label{coefficient-compare}
&&\int_{\Delta_1} \int_{\Delta_2}  \Big( 
   \sum_{F  \in 
 \mathscr{P}} \sum_{v \in \Omega_{F}^{(2)}}  \frac{|z^{\dep_v}|^2}{\|z^{\dep_v}\|^{2}_{L^2(\partial \mathbb B^d, \sigma)}}  g_{v,F}  \Big) d\delta_1(s) d\sigma(z) \notag \\
 &=& \int_{\Delta_1} \int_{\Delta_2}    \frac{|z^{\dep_{u}}|^2}{\|z^{\dep_{u}}\|_{L^2(\partial \mathbb B^d, \sigma)}^2}  d\delta_1(s) d\sigma(z) ~Uf \notag \\
 &\overset{\eqref{Uf}}=& \int_{\Delta_1} \int_{\Delta_2}   \frac{|z^{\dep_{u}}|^2}{\|z^{\dep_{u}}\|_{L^2(\partial \mathbb B^d, \sigma)}^2}  d\delta_1(s) d\sigma(z) ~\Big(\sum_{F  \in \mathscr{P}} \sum_{v \in \Omega_{F}^{(2)}}  g_{v,F} \Big).
\eeq
Letting $\Delta_1 =[0, 1]$,  we get
\beqn
 \int_{\Delta_2}  \Big( 
   \sum_{F  \in 
 \mathscr{P}} \sum_{v \in \Omega_{F}^{(2)}}  \frac{|z^{\dep_v}|^2}{\|z^{\dep_v}\|^{2}_{L^2(\partial \mathbb B^d, \sigma)}}  g_{v,F}  \Big)  d\sigma(z) 
=\int_{\Delta_2}   \frac{|z^{\dep_{u}}|^2}{\|z^{\dep_{u}}\|_{L^2(\partial \mathbb B^d, \sigma)}^2}   d\sigma(z) ~\Big(\sum_{F  \in \mathscr{P}} \sum_{v \in \Omega_{F}^{(2)}}  g_{v,F} \Big)
 \eeqn
 for every Borel subset $\Delta_2$ of $\partial \mathbb B^d$.
 Comparing the coefficients of nonzero $g_{v, F}$ on both sides, we obtain \eqref{depth-u-v-same}.

\begin{case2*} When $a > d$. \end{case2*}
\noindent
By the definition of the representing measure (see \eqref{prod-measure-def}) and Corollary \ref{coro-1-slice-rep}, 
 \beq 
 \label{3.10}
 \int_{\Delta_1} \int_{\Delta_2} U d\rho_{{}_{\mathscr{T}^{(1)}}}(s) \times d\nu_{{}_{\mathscr{T}^{(1)}}}(z) f = \int_{\Delta_1} \int_{\Delta_2}   w_{|\dep_{u} |}(s) \frac{|z^{\dep_{u}}|^2}{\|z^{\dep_{u}}\|_{L^2(\partial \mathbb B^d, \sigma)}^2}  dm(s) d\sigma(z) ~Uf,
 \eeq
(see \eqref{wt-fn} for the definition of the weight function $w_l(\cdot)$).
Also, by \eqref{Uf},  we obtain
\beq 
\label{3.11}
&& \int_{\Delta_1} \int_{\Delta_2} d\rho_{{}_{\mathscr{T}^{(2)}}}(s) \times d\nu_{{}_{\mathscr{T}^{(2)}}} (z)Uf \notag \\ &=& \int_{\Delta_1} \int_{\Delta_2}  \Big( 
   \sum_{F  \in 
 \mathscr{P}} \sum_{v \in \Omega_{F}^{(2)}} w_{|\dep_v|}(s) \frac{|z^{\dep_v}|^2}{\|z^{\dep_v}\|^{2}_{L^2(\partial \mathbb B^d, \sigma)}}  g_{v,F}  \Big) dm(s) d\sigma(z). \quad 
\eeq
Hence, by \eqref{3.11}, \eqref{intertwining} and \eqref{3.10}, we obtain 
\beq \label{coefficient-compare}
&&\int_{\Delta_1} \int_{\Delta_2}  \Big( 
   \sum_{F  \in 
 \mathscr{P}} \sum_{v \in \Omega_{F}^{(2)}} w_{|\dep_v|}(s) \frac{|z^{\dep_v}|^2}{\|z^{\dep_v}\|^{2}_{L^2(\partial \mathbb B^d, \sigma)}}  g_{v,F}  \Big) dm(s) d\sigma(z) \notag \\
 &=& \int_{\Delta_1} \int_{\Delta_2}   w_{|\dep_{u} |}(s) \frac{|z^{\dep_{u}}|^2}{\|z^{\dep_{u}}\|_{L^2(\partial \mathbb B^d, \sigma)}^2}  dm(s) d\sigma(z) ~Uf \notag \\
 &\overset{\eqref{Uf}}=& \int_{\Delta_1} \int_{\Delta_2}   w_{|\dep_{u} |}(s) \frac{|z^{\dep_{u}}|^2}{\|z^{\dep_{u}}\|_{L^2(\partial \mathbb B^d, \sigma)}^2}  dm(s) d\sigma(z) ~\Big(\sum_{F  \in \mathscr{P}} \sum_{v \in \Omega_{F}^{(2)}}  g_{v,F} \Big).
\eeq
Letting $\Delta_2 =\partial \mathbb{B}^d$,  we get
\beqn
\int_{\Delta_1}  \Big( 
  \sum_{F  \in 
 \mathscr{P}} \sum_{v \in \Omega_{F}^{(2)}} w_{|\dep_v|}(s) g_{v,F}  \Big) dm(s) 
  =  \int_{\Delta_1} w_{|\dep_{u} |}(s)   dm(s)  ~\Big( \sum_{F  \in \mathscr{P}} \sum_{v \in \Omega_{F}^{(2)}}  g_{v,F} \Big)
 \eeqn
 for every Borel subset $\Delta_1$ of $[0,1]$.
 Comparing the coefficients of nonzero $g_{v, F}$ on both sides, we obtain for almost every $s \in [0, 1]$, 
 \beqn
  w_{|\dep_v|}(s) = w_{|\dep_u|}(s) \quad \mbox{if~}v \in \Omega_{F}^{(2)}, ~g_{v, F} \neq 0.
 \eeqn
By \eqref{wt-fn},  $w_k \neq w_l$ as integrable functions for non-negative integers $k \neq l$. Thus
\beq 
\label{depth-u-v}
 \mbox{$g_{v, F} \neq 0$ implies that $|\dep_v| = |\dep_u|$ for all $v \in \Omega_{F}^{(2)}.$ }
 \eeq
Thus \eqref{coefficient-compare} becomes
\beqn 
&&\sum_{F  \in 
 \mathscr{P}} \sum_{\underset{|\dep_v| = |\dep_u|}{v \in \Omega_{F}^{(2)}}}  \Big( \int_{\Delta_1} \int_{\Delta_2}   w_{|\dep_u|}(s) 
    \frac{|z^{\dep_v}|^2}{\|z^{\dep_v}\|^{2}_{L^2(\partial \mathbb B^d, \sigma)}}    dm(s) d\sigma(z) \Big)g_{v,F}  \\
 &=& \sum_{F  \in \mathscr{P}} \sum_{\underset{|\dep_v| = |\dep_u|}{v \in \Omega_{F}^{(2)}}}  \Big( \int_{\Delta_1} \int_{\Delta_2}   w_{|\dep_{u} |}(s) \frac{|z^{\dep_{u}}|^2}{\|z^{\dep_{u}}\|_{L^2(\partial \mathbb B^d, \sigma)}^2}  dm(s) d\sigma(z) \Big)g_{v,F}.
\eeqn
Letting $\Delta_1 = [0,1]$ and comparing the coefficients of nonzero $g_{v, F}$, for every Borel subset $\Delta_2$ of $\partial \mathbb{B}^d$, we get  
\beqn
\int_{\Delta_2} \frac{|z^{\dep_{v}}|^2}{\|z^{\dep_{v}}\|_{L^2(\partial \mathbb B^d, \sigma)}} d\sigma(z) = \int_{\Delta_2}  \frac{|z^{\dep_{u}}|^2}{\|z^{\dep_{u}}\|_{L^2(\partial \mathbb B^d, \sigma)} }d\sigma(z) \quad \mbox{if~}v \in \Omega_{F}^{(2)}, ~g_{v, F} \neq 0,
\eeqn
where we used the fact that $\displaystyle \int_{[0, 1]}w_{l}(s)dm(s) \neq 0,$ $l \in \mathbb N.$
Thus \eqref{depth-u-v-same} holds in this case as well.

We next claim that \beq \label{claim} \dep_v=\dep_u \quad~\mbox{if~}v \in \Omega_{F}^{(2)}, ~g_{v, F} \neq 0.\eeq In case $a > d$ and $d=1$, 
the claim is trivial in view of  \eqref{depth-u-v}. Assume that $d \Ge 2.$ In view of continuity of the monomials and the fact that \eqref{depth-u-v-same} holds on a dense set, the equality in \eqref{depth-u-v-same} holds for all $z \in \partial \mathbb B^d$.
Consider $\alpha := (\alpha_1, \ldots , \alpha_d) \in \mathbb{N}^d$ given by $ \alpha_j = \min \{ {\dep_{u_j}, \dep_{v_j}} \},$ $j=1, \ldots, d.$
Thus
\beq\label{exponent}
\prod_{j=1}^{d} \big| z_{j}^{ \dep_{v_j}- \alpha_j}\big|^2=\frac{\|z^{\dep_{v}}\|_{L^2(\partial \mathbb B^d, \sigma)} }{\|z^{\dep_{u}}\|_{L^2(\partial \mathbb B^d, \sigma)} } \prod_{j=1}^{d} \big| z_{j}^{ \dep_{u_j} -  \alpha_j}\big|^2,   \quad z \in \partial{\mathbb{B}}^d,~ v \in \Omega_{F}^{(2)}, ~g_{v, F} \neq 0.\eeq 
Suppose to the contrary that $\dep_u \neq \dep_v$ for some $v \in \Omega_{F}^{(2)}$. Without loss of generality, we may assume that $\dep_{u_1} \neq \dep_{v_1}$. Let $ w = (0,w_2,\ldots,w_d) \in \partial \mathbb{B}^d$ be such that $w_j \neq 0$ for $j = 2, \ldots,d$. 
Then evaluating \eqref{exponent} at $w$, we get one side of \eqref{exponent} equal to zero, while the other side remains nonzero. This contradicts \eqref{exponent}, and hence $\dep_u =\dep_v.$ Thus the claim stands verified. 
It is now immediate from \eqref{Uf} and \eqref{claim} that
\beqn
U (\mathcal{L}^{(1)}_{u,F}) \subseteq \displaystyle \bigoplus_{H \in \mathscr P} \bigoplus_{\underset{\dep_v =\dep_u}{v\in \Omega_{H}^{(2)}}} \mathcal{L}^{(2)}_{v, H}.
\eeqn   
Note that by \eqref{phi-F-eqn} and \eqref{phi-F},  for any $H \in \mathscr P$ and $v \in \Omega^{(j)}_H~(j=1, 2)$, $k \in H$ if and only if $\dep_{v_k} \neq 0$, and therefore
\beqn \label{U-inclusion}
U (\mathcal{L}^{(1)}_{u,F}) \subseteq \displaystyle \bigoplus_{\underset{\dep_v =\dep_u}{v\in \Omega_{F}^{(2)}}} \mathcal{L}^{(2)}_{v,F}.
\eeqn 
This also yields
$$
U(\displaystyle \bigoplus_{\underset{\dep_u=\alpha}{u \in \Omega_{F}^{(1)}}} \mathcal{L}^{(1)}_{u,F}) \subseteq \displaystyle \bigoplus_{\underset{\dep_v =\alpha}{v\in \Omega_{F}^{(2)}}} \mathcal{L}^{(2)}_{v,F}
$$ for every $\alpha \in \mathbb N^d$ such that $\alpha_j \neq 0$ if and only if $j \in F$.
Applying this fact to $U^{-1},$ we obtain the desired conclusion in the first part.
The remaining part follows by applying the first part to $\alpha=0$ and $F = \emptyset$ (see the proof of Lemma \ref{lem-E-finite}).
\end{proof}
\begin{remark}
In case $a=1=d$, the conclusion in \eqref{depth-u-v-same} always holds, while \eqref{claim} does not follow from \eqref{depth-u-v-same} unlike the case $ad \neq 1$.
\end{remark}

In the proof of the main result, we also need a couple of facts related to the indexing set $\Omega_F,$ $F \in \mathscr P$. The following is
immediate from the definition of $\Omega_{\{l\}}$, $l=1, \ldots, d$:
 \beq \label{2identities}
 \mathcal G_{n-1}(\mathscr T_l) = \bigsqcup_{\underset{\dep_u=n \epsilon_l}{u \in \Omega_{\{l\}}}} \{\parent{u_l}\}, \quad
 \mathcal G_{n}(\mathscr T_l) = \bigsqcup_{\underset{\dep_u=n \epsilon_l}{u \in \Omega_{\{l\}}}} \sib{u_l}, \quad n \Ge 1.
 \eeq
In the proof of Theorem \ref{main-thm}, we also need a {\it canonical choice} for the indexing set $\Omega_F.$ Indeed, 
the choices of $\Omega_{\{j\}}$,  $j \in F$ yield the following natural choice for $\Omega_F$:
\beq
\label{c-choice}
\tilde{\Omega}_F:=\bigcap_{j \in F} \{(u_1, \ldots, u_d) \in V : u_l=\rootb_l~\mbox{if~}l \notin F, ~ u_{j}=\tilde{u}_{j} ~\mbox{for some~} \tilde{u} \in \Omega_{\{j\}}\}.
\eeq
To see that the above collection satisfies requirements of an indexing set $\Omega_F$ (as ensured by the axiom of choice), note the following:
\begin{enumerate} 
\item[$\bullet$] If $u, v \in \tilde{\Omega}_F$ such that $u \neq v$, then $\sibi{F}{u} \cap \sibi{F}{v}=\emptyset.$
\item[$\bullet$]  $\Phi_F = \displaystyle \bigsqcup_{v \in \tilde{\Omega}_F} \mathsf{sib}_F(v).$ To see this, let 
$u \in \Phi_F$, and fix $j \in F$. Note that the $d$-tuple $u^{(j)}:=(\rootb_1, \ldots, \rootb_{j-1}, u_j, \rootb_{j+1}, \ldots, \rootb_d) \in \Phi_{\{j\}}.$ However,
\beqn
\Phi_{\{j\}} = \displaystyle \bigsqcup_{v \in \Omega_{\{j\}}} \mathsf{sib}_{\{j\}}(v).
\eeqn
Thus $u^{(j)} \in \mathsf{sib}_{\{j\}}(v^{(j)})$ for some $v^{(j)} \in \Omega_{\{j\}}.$
Let $v$ denote the $d$-tuple with $j^{\tiny \mbox{th}}$ entry given by 
\beqn
v_j = \begin{cases} j^{\tiny \mbox{th}}~ \mbox{coordinate of~} v^{(j)} & \mbox{if~}j \in F, \\
\rootb_j & \mbox{if~}j \notin F.
\end{cases}
\eeqn
Then $u \in \mathsf{sib}_F(v)$ and $v \in \tilde{\Omega}_F.$

\end{enumerate}

The utility of the canonical choice of the indexing set is illustrated in the following.
\begin{lemma} Let $\mathscr T = (V,\mathcal E)$ be the directed Cartesian product of locally finite rooted directed trees 
$\mathscr T_1, \cdots, \mathscr T_d$. 
Then, for any $\alpha \in \mathbb N^d$ and a nonempty $F \in \mathscr P$,
\beq \label{c-sum}
\prod_{l \in F }\sum_{\underset{\dep_u=\alpha_l \epsilon_l}{u \in \Omega_{\{l\}}}} (\mbox{card}( \mathsf{sib}(u_l)) -1)
= \displaystyle \sum_{\underset{\dep_u=\alpha}{u \in \Omega_{F}}} \prod_{l\in F}(\mbox{card}( \mathsf{sib}(u_l)) -1). \eeq
\end{lemma}
\begin{proof}
Let $F=\{i_1, \ldots, i_k\},$
and note that
\beqn
&& \prod_{l \in F }\sum_{\underset{\dep_u=\alpha_l \epsilon_l}{u \in \Omega_{\{l\}}}} (\mbox{card}( \mathsf{sib}(u_l)) -1) \\
&=& \displaystyle \sum_{\underset{\dep_{u^{(1)}}=\alpha_{i_1} \epsilon_{i_1}}{u^{(1)} \in \Omega_{\{i_1\}}}} \ldots  \sum_{\underset{\dep_{u^{(k)}}=\alpha_{i_k} \epsilon_{i_k}}{u^{(k)} \in \Omega_{\{i_k\}}}} (\mbox{card}( \mathsf{sib}(u^{(1)}_{i_1})) -1) \ldots (\mbox{card}( \mathsf{sib}(u^{(k)}_{i_k})) -1) \\
&=& \displaystyle \sum_{\underset{\dep_u=\alpha}{u \in \tilde{\Omega}_{F}}} \prod_{l\in F}(\mbox{card}( \mathsf{sib}(u_l)) -1),
\eeqn
where $\tilde{\Omega}_F$ is the canonical choice as given in \eqref{c-choice}. The desired formula now follows from the fact that the summation on right hand side of \eqref{c-sum} is independent of the choice of $\Omega_F$.  
\end{proof}

We now complete the proof of the main result of this paper.
\begin{proof}[of Theorem \ref{main-thm}]
The implication (i) $\Rightarrow$ (ii) follows from Proposition \ref{level-eq}.
To see the implication (ii)$\Rightarrow$ (i), suppose that (ii) holds. Thus, for every $\alpha \in \mathbb{N}^{d}$ and every $F \in \mathscr P$, there exists a unitary $$U_{F,\alpha} :  \displaystyle \bigoplus_{\underset{\dep_u =\alpha}{u \in \Omega_{F}^{(1)}}} \mathcal{L}^{(1)}_{u,F} \rar \displaystyle \bigoplus_{\underset{\dep_v=\alpha}{v\in \Omega_{F}^{(2)}}} \mathcal{L}^{(2)}_{v,F}.$$ Define $U : E^{(1)} \rar E^{(2)}$ by setting $U=\displaystyle \bigoplus_{\underset{\alpha \in \mathbb N^d}{F \in \mathscr P }}U_{F, \alpha}$. 
Let $\rho_{{}_{\mathscr{T}^{(j)}}} \times d\nu_{{}_{\mathscr{T}^{(j)}}}$ be the representing measure of $\mathscr H_{\mf c_a}(\mathscr T^{(j)})$ (resp. $\mathscr H^{\mf s}_{\mf c_a}(\mathscr T^{(j)})$) in case $a \Ge d$ (resp. $a < d$), $j=1, 2.$ Fix  $F \in \mathscr P$ and $\alpha \in \mathbb N^d$.
Note that by \eqref{prod-measure-def}, for any Borel subset $A$ of $[0, 1] \times \partial \mathbb B^d$,
\beqn 
&& Ud\rho_{{}_{\mathscr{T}^{(1)}}} \times d\nu_{{}_{\mathscr{T}^{(1)}}}(A)\Big(\sum_{\underset{\dep_u=\alpha}{u \in \Omega^{(1)}_{F}}}  g_{u,F} \Big) 
\\ &=& \int_A \Big( 
 \sum_{\underset{\dep_u=\alpha}{u \in \Omega^{(1)}_{F}}} w_{|\dep_u|}(s) \frac{|z^{\dep_u}|^2}{\|z^{\dep_u}\|^{2}_{L^2(\partial \mathbb B^d, \sigma)}}  U_{F, \alpha}(g_{u,F})  \Big) dm(s) d\sigma(z) \\ &=& \int_A 
  w_{|\alpha|}(s) \frac{|z^{\alpha}|^2}{\|z^{\alpha}\|^{2}_{L^2(\partial \mathbb B^d, \sigma)}}   \Big( \sum_{\underset{\dep_u=\alpha}{u \in \Omega^{(1)}_{F}}} U_{F, \alpha}(g_{u,F})  \Big) dm(s) d\sigma(z) 
 \\ &=& d\rho_{{}_{\mathscr{T}^{(2)}}} \times d\nu_{{}_{\mathscr{T}^{(2)}}}(A)~U\Big(\sum_{\underset{\dep_u=\alpha}{u \in \Omega^{(1)}_{F}}}  g_{u,F} \Big).
 \eeqn
Hence, by Lemma \ref{set-measure-lem}, the Hilbert modules $\mathscr H_{\mf c_a}(\mathscr T^{(1)})$ and $\mathscr H_{\mf c_a}(\mathscr T^{(2)})$ (resp. $\mathscr H^{\mf s}_{\mf c_a}(\mathscr T^{(1)})$ and $\mathscr H^{\mf s}_{\mf c_a}(\mathscr T^{(2)})$) are isomorphic in case $a \Ge d$ (resp. $a < d$). 
By the first paragraph of the proof of Proposition \ref{level-eq}, in both these cases, 
 $\mathscr H_{\mf c_a}(\mathscr T^{(1)})$ and $\mathscr H_{\mf c_a}(\mathscr T^{(2)})$ are isomorphic.
 
 We now see the equivalence of (iii) and (iv). 
Note that for every integer $n \Ge 1$,  
\beqn
\sum_{\underset{\dep_u=n \epsilon_k}{u \in \Omega^{(j)}_{\{l\}}}} (\mbox{card}( \mathsf{sib}_l(u)) -1)  &=& \sum_{\underset{\dep_u=n \epsilon_l}{u \in \Omega^{(j)}_{\{l\}}}} \big(\mbox{card}( \mathsf{sib}(u_l)) -\mbox{card}(\parent{u_l}\big) 
 \\ &\overset{\eqref{2identities}}=& \mbox{card}(\mathcal G_{n}(\mathscr T^{(j)}_l))- \mbox{card}(\mathcal G_{n-1}(\mathscr T^{(j)}_l)).
 \eeqn
Since each $\mathscr T^{(j)}_l$ is rooted, $\mbox{card}(\mathcal G_{0}(\mathscr T^{(j)}_l))=1$. A routine telescopic sum argument now establishes the equivalence of (iii) and (iv).

To see (iii) $\Rightarrow$ (ii), 
note first that (ii) holds trivially in case $F=\emptyset.$
Let $ \alpha \in \mathbb{N}^d$ and  $F \in \mathscr P$ be nonempty. 
By (iii),  
\beqn
\prod_{l \in F }\sum_{\underset{\dep_u=\alpha_l \epsilon_l}{u \in \Omega^{(j)}_{\{l\}}}} (\mbox{card}( \mathsf{sib}_l(u)) -1) &=& 
\prod_{l \in F }\sum_{\underset{\dep_u=\alpha_l \epsilon_l}{u \in \Omega^{(j)}_{\{l\}}}} (\mbox{card}( \mathsf{sib}(u_l)) -1)
\\ &\overset{\eqref{c-sum}}=& \displaystyle \sum_{\underset{\dep_u=\alpha}{u \in \Omega_{F}^{(j)}}} \prod_{l \in F}(\mbox{card}( \mathsf{sib}(u_l)) -1).
\eeqn
However, by \eqref{dim-formula}, 
\beq \label{above-eq} \displaystyle \sum_{\underset{\dep_u=\alpha}{u \in \Omega_{F}^{(j)}}} \dim\mathcal{L}^{(j)}_{u,F} = \sum_{\underset{\dep_u=\alpha}{u \in \Omega_{F}^{(j)}}} \prod_{l \in F} (\mbox{card}(\sib{u_{l}})-1),
\eeq
which establishes (ii).
To complete the proof of the theorem, it now suffices to see that (ii) $\Rightarrow$ (iii). This follows immediately by taking $F = \{l\}$ in \eqref{above-eq}.
\end{proof}


We conclude this paper with some possibilities for further  investigations. 
It is clear that the existence of the operator-valued representing measures for Drury-Arveson-type modules or its spherical Cauchy dual modules is one of the crucial ingredients for the proof of Theorem \ref{main-thm}. In case the parameter $a$ is non-integral, we do not know whether or not the Hilbert modules $\mathscr H_{\mf c_a}(\mathscr T)$ or $\mathscr H^{\mf s}_{\mf c_a}(\mathscr T)$ admit representing measures ? 
Further, the following classification problem arises naturally in the realm of graph-theoretic operator theory. 

\begin{problem}
For $j=1,2$, let $\mathscr T^{(j)}= (V^{(j)}, 
\mathcal{E}^{(j)})$ denote the directed Cartesian product of 
locally finite, leafless, rooted directed trees 
$\mathscr T^{(j)}_1, \ldots, \mathscr T^{(j)}_d$ of finite joint branching index, and
consider 
the (graded) submodules $\mathscr N^{(j)}$ of 
the Drury-Arveson-type Hilbert module 
$\mathscr{H}_{\mf c_a}(\mathscr T^{(j)})$ generated by (homogeneous) polynomials $p_1, \ldots, p_l \in \mathbb C[z_1, \ldots, z_d]$. Under what conditions on $\mathscr T^{(1)}$ and $\mathscr T^{(2)}$,  the submodules $\mathscr N^{(1)}$ and $\mathscr N^{(2)}$ are isomorphic ? 
\end{problem}

\oneappendix
\section{Constant on parents is constant on generations}

The main result of this appendix is a rigidity theorem showing that in higher dimensions ($d \Ge 2$) the conditions \eqref{constant-gen} and \eqref{w-sp} are equivalent. Here is the precise statement.

\begin{thmA*} 
Let $\mathscr T = (V,\mathcal E)$ be the directed Cartesian product of leafless, rooted directed trees $\mathscr T_1, \ldots, \mathscr T_d$ and let $S_{\lambdab}=(S_1, \ldots, S_d)$ be a commuting multishift on $\mathscr T$. Consider
the function $\mf C : V \rar (0, \infty)$ given by
\beqn  \mf C(v) := \sum_{j=1}^d \|S_j e_v\|^2, \quad v \in V. \eeqn
If $d \Ge 2$, then the following conditions are equivalent:
\begin{enumerate}
\item[(i)] 
$\mf C$ is
constant on every generation $\mathcal G_t$, $t \in  \mathbb N$. 
\item[(ii)] $\mf C$ is constant on $\mathsf{Par}(v)$ for every $v \in V^{\circ}.$
\end{enumerate}
\end{thmA*}
Recall the following notations from \cite{CPT}: 
\beqn \mathsf{Chi}(v):=\bigcup_{j=1}^d \childi{j}{v},
~\mathsf{Par}(v):=\bigcup_{j=1}^d \parenti{j}{v}, \quad v \in V.\eeqn
In the proof of the above theorem, we need two general facts.

\begin{lemA*} \label{V-beta-constant}
Let $\mathscr T = (V, \mathcal E)$ be the directed Cartesian product of leafless, rooted directed trees $\mathscr T_1,  \ldots, \mathscr T_d$. Set
\beq \label{V-beta}
V_\beta := \{v \in V : \dep_v = \beta\}, \quad \beta \in \mathbb N^d.
\eeq
Let $f$ be a complex valued function on $V$ such that $f$ is constant on each of the sets $\child{v}$, $v \in V$. If $d \Ge 2$, then the following statements are equivalent:
\begin{itemize}
\item[(i)] $f$ is constant on each of the sets $V_{t\epsilon_1}$, $t \in \mathbb N$.
\item[(ii)] $f$ is constant on each of the sets $V_\beta$ and $f(V_\beta) = f(V_{|\beta|\epsilon_1})$, $\beta \in \mathbb N^d$.
\item[(iii)] $f$ is constant on each of the generations $\mathcal G_t$, $t \in  \mathbb N$.
\end{itemize}
\end{lemA*}

\begin{proof}
Assume that $d \Ge 2.$
Clearly, the implications (iii) $\Rightarrow$ (ii), and (ii) $\Rightarrow$ (i) hold.
To see that (i) $\Rightarrow$ (ii), 
let $\beta = (\beta_1, \ldots, \beta_d) \in \mathbb N^d$. By (i), $f$ is constant on $V_{|\beta|\epsilon_1}$. Consider the sequence $\{\Gamma_l\}_{l=1}^{d-1}$   given by
\beqn
\Gamma_1 &:=& \big\{\gamma^{(j)}_1 = (|\beta|-j, j, 0, \ldots, 0) \in \mathbb N^d :  j=1, \ldots, \beta_2 \big\},\\
\Gamma_k &:=& \big\{\gamma^{(j)}_k = 
(|\beta|-\beta_2 - \cdots - \beta_k -j, \beta_2, \ldots, \beta_{k}, j, \underbrace{0, \ldots, 0}_{\underset{\tiny \mbox{entries}}{\tiny{d-k-1}}}) \in \mathbb N^d :  \underset{\mbox{$k~=~2, \ldots, d-1.$}}{\underset{}{j=1, \ldots, \beta_{k+1}}} \big\}  \eeqn
Let $v \in V_{\gamma^{(1)}_1}$. Then $v \in \mathsf{Chi}_2(\parenti{2}{v})$ and $\mathsf{Chi}_1(\parenti{2}{v}) \subseteq V_{|\beta|\epsilon_1}$. 
Thus $\child{\parenti{2}{v}}$ intersects with $V_{|\beta| \epsilon_1}$.
Since $f$ is constant on $\child{\parenti{2}{v}}$, it follows that $f(v)$ is equal to the constant value of $f$ on $V_{|\beta| \epsilon_1}$. Since $v$ was chosen arbitrarily, we get that $f$ is constant on $V_{\gamma^{(1)}_1}$ and $f(V_{\gamma^{(1)}_1}) = f(V_{|\beta|\epsilon_1})$. 

We claim that for any $k=1, \ldots, d-1,$ if $f$ is constant on $V_{\gamma^{(j)}_k}$, then it is also constant on $V_{\gamma^{(j+1)}_k}$ (if $j < \beta_{k+1}$) or $V_{\gamma^{(1)}_{k+1}}$ (if $j=\beta_{k+1}$) with the  same constant value. We divide this verification into two cases:

\begin{case1*} 
When $j < \beta_{k+1}$.
\end{case1*}
\noindent
Note that 
\beqn
\gamma^{(j)}_k &=& (|\beta|-\beta_2 - \cdots - \beta_k - j, \beta_2, \ldots, \beta_k, j, 0, \ldots, 0), \\ \gamma^{(j+1)}_k &=& (|\beta|-\beta_2 - \cdots - \beta_k - j-1, \beta_2, \ldots, \beta_k, j+1, 0, \ldots, 0).
\eeqn
Suppose $f$ is constant on $V_{\gamma^{(j)}_k}$. Let $v \in V_{\gamma^{(j+1)}_k}$. Then $v \in \mathsf{Chi}_{k+1}(\parenti{k+1}{v})$ and $\mathsf{Chi}_1(\parenti{k+1}{v}) \subseteq V_{\gamma^{(j)}_k}$. 
Thus $\child{\parenti{k+1}{v}}$ intersects with $V_{\gamma^{(j)}_k}$.
Since $f$ is constant on $\child{\parenti{k+1}{v}}$, it follows that $f(v)$ is equal to the constant value of $f$ on $V_{\gamma^{(j)}_k}$. Since $v$ was chosen arbitrarily, we get that $f$ is constant on $V_{\gamma^{(j+1)}}$ and $f(V_{\gamma^{(j+1)}_k}) = f(V_{\gamma^{(j)}_k})$.

\begin{case2*}  
When $j = \beta_{k+1}$.
\end{case2*}
\noindent
Note that 
\beqn
\gamma^{(\beta_{k+1})}_k &=& (|\beta|-\beta_2 - \cdots - \beta_{k+1}, \beta_2, \ldots, \beta_{k+1}, 0, \ldots, 0), \\ \gamma^{(1)}_{k+1} &=& (|\beta|-\beta_2 - \cdots - \beta_{k+1} -1, \beta_2, \ldots, \beta_{k+1}, 1, 0, \ldots, 0).
\eeqn
Suppose $f$ is constant on $V_{\gamma^{(\beta_{k+1})}_k}$. Let $v \in V_{\gamma^{(1)}_{k+1}}$. Then $v \in \mathsf{Chi}_{k+2}(\parenti{k+2}{v})$ and $\mathsf{Chi}_1(\parenti{k+2}{v}) \subseteq V_{\gamma^{(\beta_{k+1})}_k}$. 
Thus $\child{\parenti{k+2}{v}}$ intersects with $V_{\gamma^{(\beta_{k+1})}_k}$.
Since $f$ is constant on $\child{\parenti{k+2}{v}}$, it follows that $f(v)$ is equal to the constant value of $f$ on $V_{\gamma^{(\beta_{k+1})}_k}$. Since $v$ was chosen arbitrarily, we get that $f$ is constant on $V_{\gamma^{(1)}_{k+1}}$ and $f(V_{\gamma^{(1)}_{k+1}}) = f(V_{\gamma^{(\beta_{k+1})}_k})$.
Thus the claim stands verified. Since $\gamma^{(\beta_d)}_{d-1}=\beta$, we obtain
\beqn
f(V_{|\beta|\epsilon_1}) = f(V_{\gamma^{(1)}_1}) = \ldots = f(V_{\gamma^{(\beta_d)}_{d-1}})=f(V_{\beta}).
\eeqn
This yields (ii). 

Now we show that (ii) $\Rightarrow$ (iii). Let $u,v$ be any two vertices in $\mathcal G_t$, $t \in \mathbb N$. Then $|\dep_u| = |\dep_v| = t$ and by (ii), $f(V_{\dep_v}) = f(V_{\dep_u}) = f(V_{t\epsilon_1})$. This shows that $f(u) = f(v)$, which proves (iii). 
\end{proof}

\begin{lemB*}
Let $\mathscr T = (V, \mathcal E)$ be the directed Cartesian product of leafless, rooted directed trees $\mathscr T_1,  \ldots, \mathscr T_d$, $d \Ge 2$ and let $f$ be a complex valued function on $V$. Then the following statements are equivalent:
\begin{itemize}
\item[(i)] $f$ is constant on each of the sets $\child{v}$, $v \in V$.
\item[(ii)] $f$ is constant on each of the sets $\mathsf{Par}(v)$, $v \in V^\circ$.
\end{itemize}
\end{lemB*}

\begin{proof}
To see that (i) $\Rightarrow$ (ii), let $f$ be constant on each of the sets $\child{v}$, $v \in V$. Let $v \in V^\circ$ and $u, w \in \mathsf{Par}(v)$. Then $u = \parenti{i}{v}$ and $w = \parenti{j}{v}$ for some $1 \Le i, j \Le d$. Note that
$$u \in \childi{j}{\mathsf{par}_i\mathsf{par}_j(v)} \mbox{ and } w \in \childi{i}{\mathsf{par}_i\mathsf{par}_j(v)}.$$
Thus $u, w \in \child{\mathsf{par}_i\mathsf{par}_j(v)}$. Hence, by the hypothesis, $f(u) = f(w)$, proving (ii).

To see that (ii) $\Rightarrow$ (i), let $f$ be constant on each of the sets $\mathsf{Par}(v)$, $v \in V^\circ$. Let $v \in V$ and $u, w \in \child{v}$. Then $u \in \childi{i}{v}$ and $w \in \childi{j}{v}$ for some $1 \Le i, j \Le d$. 

\begin{case1*} 
When $i \neq j$. 
\end{case1*}

Without loss of generality, assume that $i < j$. In this case, consider the vertex $$\eta = (v_1, \ldots, u_i, \ldots, w_j, \ldots, v_d).$$ Note that $u_i \in \child{v_i}$, $w_j \in \child{v_j}$, $u = \parenti{j}{\eta}$ and $w = \parenti{i}{\eta}$. Thus $u, w \in \mathsf{Par}(\eta)$. Hence, by (ii), $f(u) = f(w)$. 

\begin{case2*} 
When $i = j$. 
\end{case2*}
For any positive integer $k \in \{1, \ldots, d\}$ such that $k \neq i$, consider the vertices $$\theta = (v_1, \ldots, \eta_k, \ldots, u_i, \ldots, v_d) \mbox{ and } \xi = (v_1, \ldots, \eta_k, \ldots, w_i, \ldots, v_d),$$
where $\eta_k \in \child{v_k}.$ Note that $u_i, w_i \in \child{v_i}$, and $\parenti{i}{\theta} = \parenti{i}{\xi}$. Hence $\mathsf{Par}(\theta) \cap \mathsf{Par}(\xi) \neq \emptyset$. Further, $u = \parenti{k}{\theta} \in \mathsf{Par}(\theta)$ and $w = \parenti{k}{\xi} \in \mathsf{Par}(\xi)$. Since $f$ is constant on $\mathsf{Par}(\theta)$ as well as on $\mathsf{Par}(\xi)$, and $\mathsf{Par}(\theta),$ $\mathsf{Par}(\xi)$ have a common vertex, it follows that $f(u) = f(w)$.

This proves (i).
\end{proof}

\begin{proof}[of Theorem A]
Assume that $d \Ge 2.$
In view of Lemmas A and B (applied to $f=\mf C$), it suffices to show that if \beq \label{A.2} \mbox{$\mf C$ is constant on each of the sets $\child{v}$, $v \in V$,}\eeq 
then it is constant on $V_{t \epsilon_1}$ for every $t \in \mathbb N$, where $V_{\beta}$ is as defined in \eqref{V-beta}. To this end, let $t \in \mathbb N$ and let $u , v$ be any two vertices in $V_{t\epsilon_1}$. We need to show that $\mf C(u) = \mf C(v)$. Observe that as $\dep_u  = t\epsilon_1 = \dep_v$, we must have $u = (u_1, \rootb_2, \cdots, \rootb_d)$ and $v = (v_1, \rootb_2, \cdots, \rootb_d)$ for some vertices $u_1, v_1$ of depth $t$  in $\mathscr T_1$. Let $k$ denote the unique least non-negative integer such that 
\beq \label{v-u equal}
\parentn{k}{u_1} = \parentn{k}{v_1}.
\eeq
If $k=0$, then $u= v$, and hence $\mf C(u)=\mf C(v)$ holds trivially. So assume that $k \Ge 1$. Consider the sequence $\{v_2^{(l)}\}_{l \in \mathbb N}$ of vertices in $V_2$ with the following conditions:
\beqn
\parent{v_2^{(1)}} = \rootb_2 = v^{(0)}_2, ~ \parent{v_2^{(l)}} = v_2^{(l-1)}, ~ l \Ge 2.
\eeqn
Now consider the sequence $\{v^{(l)}\}_{l =1}^{k}$ of vertices in $\mathscr T$ given as follows:
\beqn
v^{(l)} &=& \big(\parentn{l}{v_1}, v_2^{(l-1)}, \rootb_3, \ldots, \rootb_d \big), ~ l =1, \ldots, k.
\eeqn
Further, consider the sequence $\{\theta^{(l)}\}_{l = 1}^{k}$ of vertices in $\mathscr T$ given as follows:
\beqn
\theta^{(l)} &=& \big(\parentn{l}{v_1}, v_2^{(l)}, \rootb_3, \ldots, \rootb_d \big),~ l =1, \ldots, k.
\eeqn
Notice that $v \in \childi{1}{v^{(1)}}$ and $\theta^{(1)} \in \childi{2}{v^{(1)}}$. Thus $v, \theta^{(1)} \in \child{v^{(1)}}$, and hence by \eqref{A.2}, $\mf C(v) = \mf C(\theta^{(1)})$. Further, observe that $\theta^{(1)} \in \childi{1}{v^{(2)}}$ and $\theta^{(2)} \in \childi{2}{v^{(2)}}$. Once again, by \eqref{A.2}, $\mf C(\theta^{(1)}) = \mf C(\theta^{(2)})$. 
A finite inductive argument together with \eqref{A.2} shows 
$\mf C(\theta^{(l-1)}) = \mf C(\theta^{(l)})$ for $l=2, \ldots, k-1.$
Thus we obtain
\beq\label{f-theta-k1}
\mf C(v) = \mf C(\theta^{(1)}) = \ldots = \mf C(\theta^{(k)}).
\eeq
Note that by \eqref{v-u equal}, 
\beqn
v^{(k)} = \big(\parentn{k}{v_1}, v_2^{(k-1)}, \rootb_3, \ldots, \rootb_d \big) = \big(\parentn{k}{u_1}, v_2^{(k-1)}, \rootb_3, \ldots, \rootb_d \big).
\eeqn
Now consider the sequence $\{w^{(l)}\}_{l =1}^{k-1}$ of vertices in $\mathscr T$ given as follows:
\beqn
w^{(l)} &=& \big(\parentn{k-l}{u_1}, v_2^{(k-l-1)}, \rootb_3, \ldots, \rootb_d \big), \quad l =1, \ldots, k-1.
\eeqn
Further, consider the sequence $\{\eta^{(l)}\}_{l =1}^{k}$ of vertices in $\mathscr T$ given as follows:
\beqn
\eta^{(l)} &=& \big(\parentn{k-l}{u_1}, v_2^{(k-l)}, \rootb_3, \ldots, \rootb_d \big), \quad l =1, \ldots, k.
\eeqn
Observe that $\eta^{(1)} \in \childi{1}{v^{(k)}}$ and $\theta^{(k)} \in \childi{2}{v^{(k)}}$. Thus $\eta^{(1)}, \theta^{(k)} \in \child{v^{(k)}}$, and hence by \eqref{A.2}, $\mf C(\theta^{(k)}) = \mf C(\eta^{(1)})$. Further, observe that $\eta^{(1)} \in \childi{2}{w^{(1)}}$ and $\eta^{(2)} \in \childi{1}{w^{(1)}}$. 
Arguing as above, we have $\mf C(\eta^{(1)}) = \mf C(\eta^{(2)})$. A finite inductive argument now shows that
\beq\label{f-theta-k1-eta}
\mf C(\theta^{(k)}) = \mf C(\eta^{(1)}) = \ldots = \mf C(\eta^{(k)}) = \mf C(u).
\eeq
Combining \eqref{f-theta-k1} and \eqref{f-theta-k1-eta}, we get $\mf C(v) = \mf C(u)$. 
\end{proof}

\begin{acknowledgements}
The authors convey their sincere thanks to Rajeev Gupta for some stimulating conversations pertaining to the dimension formula of Section 2.1. 
\end{acknowledgements}

\affiliationthree{
  S. Chavan, D. K. Pradhan {\small \&}  S. Trivedi \\
   Department of Mathematics and Statistics\\
Indian Institute of Technology Kanpur\\
Kanpur 208016\\
   India
   \email{chavan@iitk.ac.in \\ dpradhan@iitk.ac.in \\ shailtr@iitk.ac.in}}


\begin{thebibliography}{9}



\bibitem{AV} D. Alpay and D. Volok, 
Point evaluation and Hardy space on a homogeneous tree, {\it
Integr. Eqn. Op. Th.} {\bf 53} (2005), 1-22. 

 
 
%
%
%
%
%
%

\bibitem{AC} A. Anand and S. Chavan, A moment problem and joint q-isometry tuples, {\it Complex Anal. Oper. Theory}, {\bf 11} (2017), 785-810. 

\bibitem{ACJS} A. Anand, S. Chavan, Z. Jab{\l}o\'nski, and J. Stochel,
{A solution to the Cauchy dual subnormality problem for 
$2$-isometries}, preprint, 2017.  arXiv:1702.01264v2 [math.FA]
%
%
%

\bibitem{AFFP} J. Aramayona, J. Fern\'{a}ndez, P. Fern\'{a}ndez, and C. Mart\'{i}nez-P\'{e}rez,  Trees, homology, and automorphism groups of RAAGs, preprint, 2017.  	arXiv:1707.02481 [math.GR]


\bibitem{Ar} W. Arveson, Subalgebras of $C^*$-algebras. III. Multivariable operator theory,
{\it Acta Math.} {\bf 181} (1998), 159-228.
%
%
%
%
%
%
%
%
%
%
%

%
%
%
\bibitem{BJJS} P. Budzy\'{n}ski, Z. Jab{\l}o\'nski, Il Bong Jung, and J. Stochel, 
Unbounded subnormal composition operators in $L^2$-spaces, {\it J. Funct. Anal.} {\bf 269} (2015), 2110-2164.

%
%

\bibitem{BDPP} P. Budzy\'{n}ski,  P. Dymek, A. Planeta and M. Ptak, Weighted shifts on directed trees. Their multiplier algebras, reflexivity and decompositions, preprint, 2017. arXiv:1702.00765 [math.FA]

%
%

%

\bibitem{CPT} S. Chavan, D. Pradhan and S. Trivedi, Multishifts on directed Cartesian product of rooted directed trees,  {\it Dissertationes Mathematicae}, to appear.

\bibitem{CPT-1} S. Chavan, D. Pradhan and S. Trivedi, Dirichlet spaces associated with locally finite rooted directed trees, {\it
Integr. Eqn. Op. Th}, to appear. 

%
%

\bibitem{CT}
S. Chavan and S. Trivedi, An analytic model for left-invertible
weighted shifts on directed trees, {\it J. London Math. Soc.} {\bf 94} (2016), 253-279.

\bibitem{CY}
S. Chavan and D. Yakubovich, Spherical tuples of Hilbert space operators, 
{\it Indiana Univ. Math. J.} {\bf 64} (2015), 577-612.
%
%
%
%
%
%
\bibitem{Co} J. Conway, {\it The Theory of Subnormal Operators}, Math. Surveys Monographs, 36, Amer. Math. Soc. Providence, RI 1991.


\bibitem{C-Y} R. Curto and J. Yoon, 
Disintegration-of-measure techniques for commuting multivariable weighted shifts, {\it Proc. London Math. Soc.} {\bf 92} (2006), 381-402.
 
%
%
%
%
%
%
%
%
%
%
%
%
%
%
%
%
%
%

\bibitem{Dr} S. Drury, {A generalization of von Neumann's inequality
to the complex ball}, {\it Proc. Amer. Math. Soc.} \textbf{68} (1978), 300-304.


%
%
%
%
%
\bibitem{GR} J. Gleason and S. Richter, {m-Isometric commuting tuples of operators on a Hilbert space},
{\it Integr. Eqn. Op. Th.} {\bf 56} (2006), 181-196.

%
%
%

\bibitem{H} B. Hall, {\it Lie Groups, Lie Algebras, and Representations, An Elementary Introduction}. Second edition. Graduate Texts in Mathematics, 222, Springer, Cham, 2015.

%
%
%
%
%
\bibitem{JJS}
Z. Jab{\l}o\'nski, Il Bong Jung, and J. Stochel, Weighted shifts on directed
trees, {\it  Mem. Amer. Math. Soc.} {\bf 216} (2012), {no 1017}, viii+106.


%
%
%
%
%
\bibitem{JL}
N. P. Jewell and A. R. Lubin, Commuting weighted shifts and analytic function theory in several variables, {\it J. Operator Theory}, {\bf 1} (1979), 207-223. 

%
%
%
\bibitem{KK} E. Katsoulis and D. Kribs, 
Isomorphisms of algebras associated with directed graphs,
{\it Math. Ann.} {\bf 330} (2004), 709-728.
 
%
%
%
%
%
%
%
%
%
%
%
%
%
%
%
%



\bibitem{LC} W. Light and E. Cheney, 
{\it Approximation Theory in Tensor Product Spaces},
Lecture Notes in Mathematics, 1169. Springer-Verlag, Berlin, 1985. 

\bibitem{N} Y. Neretin, Groups of hierarchomorphisms of trees and related Hilbert spaces, {\it J. Funct. Anal.} {\bf 200} (2003), 505-535. 

\bibitem{PR}  V. Paulsen and M. Raghupathi, {\it An Introduction to the Theory of Reproducing Kernel Hilbert Spaces}, Cambridge Studies in Advanced Mathematics, 152. Cambridge University Press, Cambridge, 2016.

\bibitem{Sa} J. Sarkar, {\it Applications of Hilbert Module Approach to Multivariable Operator Theory} (Survey article) Handbook of Operator Theory, Springer (2015), 1035-1091 (Edited by D. Alpay). 

\bibitem{Sh}
S. Shimorin, Wold-type decompositions and wandering subspaces for operators
close to isometries, {\it J. Reine Angew. Math.} {\bf 531} (2001), 147-189.

\bibitem{Si}  B. Simon, {\it Real analysis. A Comprehensive Course in Analysis}, Part 1. American Mathematical Society, Providence, RI, 2015.

\bibitem{S} B. Solel, 
You can see the arrows in a quiver operator algebra, 
{\it J. Aust. Math. Soc.} {\bf 77} (2004), 111-122.
 
%
%
%
%
%

\bibitem{Z}  K. Zhu, {\it Spaces of Holomorphic Functions in the Unit Ball}, Graduate Texts in Mathematics, 226. Springer-Verlag, New York, 2005.
\end{thebibliography}
\end{document}